\documentclass[12pt]{amsart}
\usepackage{amsmath, amssymb}
\usepackage{array}
\usepackage[frame,cmtip,arrow,matrix,line,graph,curve]{xy}
\usepackage{graphpap, color, paralist, pstricks}
\usepackage[mathscr]{eucal}
\usepackage[pdftex]{graphicx}
\usepackage[pdftex,colorlinks,backref=page,citecolor=blue]{hyperref}
\usepackage{tikz}
\usepackage{tikz-cd}
\usepackage{enumitem}
\usepackage{comment}

\newtheorem{theorem}{Theorem}[section]
\newtheorem{theoremletter}{Theorem}

\newtheorem{proposition}[theorem]{Proposition}
\newtheorem{corollary}[theorem]{Corollary}
\newtheorem{lemma}[theorem]{Lemma}

\theoremstyle{definition}
\newtheorem{definition}[theorem]{Definition}

\newtheorem{question}[theorem]{Question}
\newtheorem{remark}[theorem]{Remark}

\newcommand{\CC}{\mathbb{C} }

\newcommand{\TT}{\mathbb{T} }

\newcommand{\ZZ}{\mathbb{Z} }

\newcommand{\cA}{\mathcal{A} }

\newcommand{\cS}{\mathcal{S} }

\newcommand{\cU}{\mathcal{U} }

\newcommand{\rL}{\mathrm{L} }
\newcommand{\rM}{\mathrm{M} }

\newcommand{\be}{\mathbf{e} }

\newcommand{\bm}{\mathbf{m} }
\newcommand{\bn}{\mathbf{n} }

\newcommand{\lt}{\mathrm{lt} }
\newcommand{\MRY}{\mathrm{MRY}}
\newcommand{\LRY}{\mathrm{LRY}}

\newcommand{\RMC}{\mathsf{B}(\Sigma, V)^{+}}
\newcommand{\TD}{\mathsf{T}}

\DeclareMathOperator{\Int}{\mathrm{Int}}

\title{Center of generalized skein algebras}
\author{Hiroaki Karuo}
\address{Department of Mathematics, Gakushuin University, Mejiro, Toshima-ku, Tokyo, Japan}
\email{hiroaki.karuo@gakushuin.ac.jp}

\author{Han-Bom Moon}
\address{Department of Mathematics, Fordham University, New York, NY 10023}
\email{hmoon8@fordham.edu}

\author{Helen Wong}
\address{Department of Mathematical Sciences, Claremont McKenna College,
Claremont, CA 91711}
\email{hwong@cmc.edu}

\date{\today}

\begin{document}
\maketitle
\begin{abstract}
We consider a generalization of the Kauffman bracket skein algebra of a surface that is generated by loops and arcs between marked points on the interior or boundary, up to skein relations defined by Muller and Roger-Yang.   We compute the center of the Muller-Roger-Yang skein algebra  and show that this algebra is almost Azumaya when the quantum parameter $q$ is a primitive $n$-th root of unity with odd $n$. We also discuss the implications on the representation theory of the Muller-Roger-Yang  generalized skein algebra.
\end{abstract}

\section{Introduction}\label{sec:intro}
Since its introduction \cite{Prz91, Tur91}, the Kauffman bracket skein algebra of a surface $\Sigma$  has been studied for its rich connection to many areas of low-dimensional topology, including knot theory, hyperbolic geometry through the character variety, and topological quantum field theory.  In the course of these studies, many  generalizations of the skein algebra have emerged in the past decade \cite{RY14, Mul16, Le18, BKL24}.  Current research explores both the inter-connected relationships between the various generalizations and with the many fields of mathematics related to them.

In this paper, our main object is a generalization that comes from deformation quantization of the decorated Teichm\"uller space of Penner \cite{Pen87}.    This \emph{Muller-Roger-Yang generalized skein algebra} $\cS_{q}^{\MRY}(\Sigma)$ was introduced in \cite{BKL24},  based on earlier works of Roger-Yang \cite{RY14} and Muller \cite{Mul16}.   It is generated by embedded loop classes in $\Sigma$ as well as arcs with endpoints at either interior punctures or marked points on the boundary of $\Sigma$, and there are extra skein relations for arcs whose endpoints meet.     Our results concern the algebraic structure of $\cS_{q}^{\MRY}(\Sigma)$, which is an important step towards understanding its representation theory and further connections with hyperbolic geometry.

A reason why we are particularly interested in $\cS_{q}^{\MRY}(\Sigma)$ is its natural connection with cluster algebras \cite{FZ02} from combinatorial algebraic geometry.  More specifically, for any triangulable oriented surface $\Sigma$, the cluster algebra $\cA(\Sigma)$ is a combinatorial commutative algebra  generated by arc classes \cite{FST08, FT18}, and it is natural to consider its quantization. A quantum cluster algebra $\cA_{q}(\Sigma)$, which is a deformation quantization of $\cA(\Sigma)$, can be defined \cite{BZ05} for a surface without any interior punctures, and  it is identical up to localization with  Muller's generalization of the skein algebra  $\cS_{q}^{\rM}(\Sigma)$  \cite{Mul16}. However, $\cA_{q}(\Sigma)$ is not well-defined when there is an interior puncture. On the other hand,  when $q = 1$, there is an explicit relationship among $\cA(\Sigma)$, $\cS_{q}^{\MRY}(\Sigma)$, and another cluster algebra related to $\cA(\Sigma)$ that is called the upper cluster algebra $\cU(\Sigma)$ \cite{MW24, KMW25+}.  With this point of view, one may understand $\cS_{q}^{\MRY}(\Sigma)$ as a deformation quantization of $\cA(\Sigma)$, thus providing an alternative approach to a quantum cluster algebra even in the case of surfaces with interior punctures.

\subsection{Main results}
In this paper, our main contribution is the complete characterization of the center of the Muller-Roger-Yang skein algebra. 

Let $\Sigma$ be an oriented surface with interior punctures and marked points on its boundary, and let $\{ v_i\}$ denote the set of interior punctures.   Let  $\cS_{q}^{\MRY}(\Sigma)$ be the Muller-Roger-Yang skein algebra, which is a $\CC[v_{i}^{\pm}]$-algebra we define more precisely in Definition \ref{def:MRY}.  Let $T_{n}(x)$ be the $n$-th Chebyshev polynomial of the first kind.  


\begin{theoremletter}\label{thm:mainthm}
At a primitive root of unity $q$ of odd order $n$, the center of the Muller-Roger-Yang skein algebra $Z(\cS_{q}^{\MRY}(\Sigma))$ is the $\CC[v_{i}^{\pm}]$-subalgebra generated by the following elements. 
\begin{enumerate}
\item $T_{n}(\alpha)$, \; where $\alpha$ is a loop class  that has a diagram with no self-intersections, ;
\item  $\frac{1}{\sqrt{v}\sqrt{w}}T_{n}(\sqrt{v}\sqrt{w}\beta)$, \; where $\beta$ is an arc class connecting two distinct interior punctures $v$ and $w$ and  has a diagram with no self-intersections;
\item $\beta^{n}$, \;  where $\beta$ is an arc class with one endpoint at a boundary marked point and has a diagram with no self-intersections;
\item $\beta_{D} := \prod_{i}\beta_{i}$, \; where $D$ is a component of $\partial \Sigma$ that is the union of cyclically ordered, boundary arcs $\beta_{1}, \beta_{2}, \cdots, \beta_{k}$ whose endpoints are the marked points on $D$. 
\end{enumerate}
\end{theoremletter}

As we will explain in Remark \ref{rmk:Tnbetaisinthealgegbra},  the element $\frac{1}{\sqrt{v}\sqrt{w}}T_{n}(\sqrt{v}\sqrt{w}\beta)$ in Item (2) of Theorem~\ref{thm:mainthm} is an element of $\cS_{q}^{\MRY}(\Sigma)$, although $\sqrt{v}$ itself is  not  in $\cS_{q}^{\MRY}(\Sigma)$.    
 Also note that   for simplicity we have described the central elements here using the multiplication operation of the skein algebra.    
The central elements can also be described using a threading operation that is ubiquitous in quantum topology.
We will provide more details in Section~\ref{sec:central}.    

As special cases, $\cS_{q}^{\MRY}(\Sigma)$ reduces to the Muller skein algebra $\cS_{q}^{\mathrm{M}}(\Sigma)$  in the absence of interior punctures, and to the Roger-Yang skein algebra $\cS_{q}^{\mathrm{RY}}(\Sigma)$ in the absence of boundary marked points.  Our result agrees with Korinman's  computation of center for the Muller skein algebra \cite{Kor21}.  In the case of the Roger-Yang skein algebra, we have the following result.

\begin{corollary}\label{cor:centerRY}
Let $\Sigma$ be a punctured surface without a boundary. The center of the Roger--Yang skein algebra $\cS_{q}^{\mathrm{RY}}(\Sigma)$ is the $\CC[v_{i}^{\pm}]$-subalgebra generated by the following elements. 
\begin{enumerate}
\item For a loop class $\alpha$ that has a diagram with no self-intersections, $T_{n}(\alpha)$;
\item For an arc class $\beta$ connecting two distinct interior punctures $v$ and $w$ and has a diagram with no self-intersections, $\frac{1}{\sqrt{v}\sqrt{w}}T_{n}(\sqrt{v}\sqrt{w}\beta)$.
\end{enumerate}
\end{corollary}

The delicate part of  the proof  of Theorem \ref{thm:mainthm} involves the algebraic behavior of arcs with endpoints at interior punctures.  It relies on the observation that $\cS_{q}^{\MRY}(\Sigma)$ behaves like a `quadratic extension algebra' of the usual skein algebra. For any arc class $\beta$ joining two interior punctures $v$ and $w$, applying the puncture-skein relation twice at $v$ and $w$, one can check that $vw\beta^{2}$ is a linear combination of loops.  We call it the `square trick,' and note that there are similar results for other types of arc classes.  We will show, in Section \ref{sec:centercomputation}, that many standard techniques for the usual skein algebra, including the edge coordinates of the curve class with respect to an ideal triangulation $\Delta$ on $\Sigma$, a basis with total lexicographical order, and elimination of leading term of the given central element \cite{FKBL19, Kor21, Yu23} work very well for $\cS_{q}^{\MRY}(\Sigma)$ via the square trick. 

We also capitalize on relationships between $\cS_{q}^{\MRY}(\Sigma)$ and other generalizations of the skein algebra, specifically {\sl stated skein algebras} that involve arcs with assignment of $\pm$ at the endpoints at boundary marked points \cite{Le18, BKL24}.   These various skein algebras are defined  and their bases are described in Section \ref{sec:MRY}.  In Section \ref{sec:central}, we show that the elements in the statement of Theorem \ref{thm:mainthm} are central.  In Section \ref{ssec:coordinates}, we apply the square trick to generalize techniques for the usual skein algebra to prove they generate the center.  

While our characterization of the center focuses on the case of $n$ odd, we expect that an analogous theorem can be proven for $n$ even using a similar line of proof, though possibly with a different generating set.

\subsection{Representation theory of $\cS_{q}^{\MRY}(\Sigma)$}\label{ssec:represenation}
For the usual skein algebra generated only by loops, the representation theory is closely connected to both topological quantum field theory \cite{BHMV95} and character varieties \cite{Bul97, PS00}; see e.g. \cite{BW16, FKBL19, GJS24, KK22} for more details. 
Understanding the algebraic structure of  the center of $\cS_{q}^{\MRY}(\Sigma)$ is a crucial step towards classifying the finite-dimensional representations of $\cS_{q}^{\MRY}(\Sigma)$.   In particular, using Theorem \ref{thm:mainthm}, we show that $\cS_{q}^{\MRY}(\Sigma)$ is an almost Azumaya algebra when $\Sigma$ has at last one marked point on the boundary (Proposition \ref{prop:almostAzumaya}).

Recall that when a given algebra $A$ is an \emph{Azumaya algebra} and its center $Z(A)$ is finitely generated, the representation theory is particularly nice.  There is a bijection between the set of isomorphism classes of irreducible representations of $A$ and $\mathrm{MaxSpec}\; Z(A)$.  An algebra $A$ is called \emph{almost Azumaya} if, after taking a localization by $c \in Z(A)$, it becomes an Azumaya algebra. In that case,  there is an injective map from the Azumaya locus  $\mathrm{MaxSpec}\; Z(A)_{c}$ to the set of finite dimensional irreducible representations of $A$.

Thus, because $\cS_{q}^{\MRY}(\Sigma)$ is almost Azumaya,  there is a large class of irreducible representations that are completely determined by  points in the Azumaya locus.   It would be an interesting question to understand the hyperbolic geometric content of the irreducible representations and their connection to  representations of quantum cluster algebras in the view of (almost) Azumaya algebras \cite{MNTY24}.   More discussion about the proof of that $\cS_{q}^{\MRY}(\Sigma)$ is almost Azumaya and some representation theoretic consequences can be found in Section \ref{sec:almostAzumaya}.

\subsection*{Notation and convention}
We use $\CC$ as the coefficient ring. The parameter $q \in \CC$ is a primitive $n$-th root of unity, where $n \ge 3$ is a positive odd integer.

\subsection*{Acknowledgements}
 We are grateful to the  American Institute of Mathematics, whose Quantum Invariants and Low-Dimensional Topology workshop put together this collaboration.   We also thank Thang T. Q. L\^e for pointing out an alternative proof strategy (see Remark \ref{rmk:approachviaorderlyfinitegeneration}) 
and Francis Bonahon for feedback on an earlier draft.  We especially thank the anonymous referee for a careful reading of the manuscript and their helpful suggestions. 

HK was supported by JSPS KAKENHI Grant Number JP23K12976. HW was partially supported by   
DMS-2305414 from the US National Science Foundation.

\section{Muller-Roger-Yang skein algebra}\label{sec:MRY}
In this section, we fix some notation and define several versions of generalized skein algebras which include arc classes over a surface with a boundary.

\subsection{Marked surfaces and tangles}
A \emph{marked surface} $\Sigma = (\underline{\Sigma}, V)$ consists of a compact surface $\underline{\Sigma}$ with (possibly empty) boundary and a finite set of points $V \subset \underline{\Sigma}$. In the literature, a marked surface is also referred to as a punctured bordered surface. 
 In the following, we will use $\partial \Sigma:= \partial \underline{\Sigma}$ and $\Int \Sigma :=  \Int \underline{\Sigma} \setminus V$.
The set $V$ is decomposed into $V = V_{\circ} \sqcup V_{\partial}$, where $V_{\partial} = V \cap \partial \Sigma$ and $V_{\circ} = V \setminus V_{\partial}$. $V_{\partial}$ is called the set of \emph{(boundary) marked points} and $V_{\circ}$ is the set of \emph{(interior) punctures}. Each connected component of $\partial\Sigma \setminus V_\partial$ is called a \emph{boundary edge}. If the context is clear, we will not distinguish between $\Sigma$ and $\underline{\Sigma}$.

A \emph{$V$-tangle} is an embedded unoriented $1$-dimensional compact submanifold $\alpha$ of $\Sigma\times (-1,1)$ equipped with a framing (a choice of nowhere vanishing section of the normal bundle) such that

\begin{enumerate}
\item $\partial \alpha \subset V \times (-1,1)$, 
\item $\Int \alpha \subset \Int \Sigma\times (-1,1)$.
\end{enumerate}

For $(x,t)\in \Sigma\times (-1,1)$, the \emph{height} of $(x,t)$ is $t$. 

A \emph{$\partial$-tangle} is an embedded unoriented $1$-dimensional compact submanifold $\alpha$ of $\Sigma\times (-1,1)$ equipped with a framing such that \begin{enumerate}
\item $\partial \alpha \subset ((\partial\Sigma \setminus V_{\partial})\cup V_{\circ}) \times (-1,1)$, 
\item for each boundary edge $e$, the elements of $\partial \alpha\cap e\times (-1,1)$ have distinct heights,
\item $\Int \alpha \subset \Int \Sigma\times (-1,1)$.
\end{enumerate}

Note that when a component of a $V$-tangle (resp. $\partial$-tangle) $\alpha$ meets $\partial \Sigma$, it hits (resp. avoids) the boundary marked points $V_{\partial}$. We call both $V$-tangles and $\partial$-tangles \emph{tangles}. 

Two $V$-tangles are \emph{isotopic} if they are isotopic in the class of $V$-tangles. We say the framing of a one-dimensional submanifold $\alpha$ of $\Sigma\times (-1,1)$ is \emph{vertical} if at each point of $\partial \alpha$, the normal vector defined from the framing points towards the direction of $1$. Every $V$-tangle is isotopic to a $V$-tangle with a vertical framing. A $V$-tangle in $\Sigma\times(-1,1)$ with a vertical framing is in a \emph{general position} if its image by the projection $\Sigma\times(-1,1)\to \Sigma\times\{0\}$ has only transversal multiple points as singular points, especially there are only transversal double points in $\Sigma\setminus (V \cup \partial \Sigma)$. For the image of a $V$-tangle in general position under projection, we assign over/under crossing information to each singular point with respect to the heights. 
The resulting diagram in $\Sigma = (\underline{\Sigma},V)$ is called a \emph{$V$-tangle diagram} of the $V$-tangle. 

We may also define the \emph{isotopy} of $\partial$-tangles, vertical $\partial$-tangles, and $\partial$-tangles in \emph{a general position} in a similar way. In the case of $\partial$-tangle diagrams, we assign the total order to the endpoints on each boundary edge according to their heights. We employ the following convention for $\partial$-tangle diagrams: If an orientation is assigned to a boundary edge in a picture, the heights of the drawn arcs incident to the edge are assumed to be increasing with respect to the orientation. However, the heights do not always increase if a boundary orientation is not drawn. 

\subsection{Skein algebras}\label{ssec:skeinalgebra}
For each interior puncture $v_{i} \in V_{\circ}$, we set a formal variable (using the same notation $v_{i}$), which commutes with everything. From now on, we use the Laurent polynomial ring $\CC[v_{i}^{\pm}]$ generated by all $v_{i}^{\pm} \in V_{\circ}$ as the coefficient ring. 

\begin{definition}[The Muller--Roger--Yang skein algebra \cite{BKL24}]\label{def:MRY}
We fix $q \in \CC^{*}$. The \emph{Muller--Roger--Yang skein algebra} of a marked surface $(\Sigma,V)$, denoted by $\cS_{q}^{\MRY}(\Sigma)$, is the $\CC[v_{i}^{\pm}]$-algebra generated by all isotopy classes of $V$-tangles in $\Sigma\times (-1,1)$, subject to 
\begin{align*}
&({\rm A})\begin{array}{c}\includegraphics[scale=0.15]{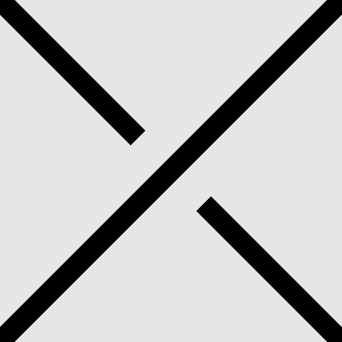}\end{array}
=q\begin{array}{c}\includegraphics[scale=0.15]{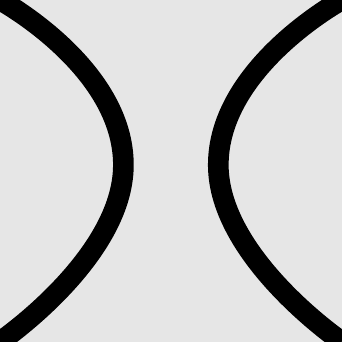}\end{array}
+q^{-1}\begin{array}{c}\includegraphics[scale=0.15]{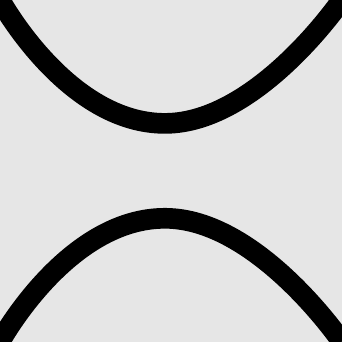}\end{array},
\\
&({\rm B})\begin{array}{c}\includegraphics[scale=0.15]{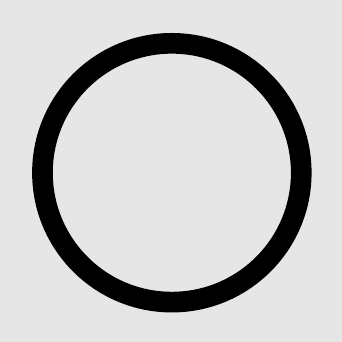}\end{array}
=(-q^2-q^{-2})\begin{array}{c}\includegraphics[scale=0.15]{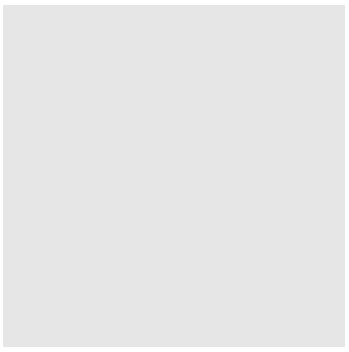}\end{array}, \\
&({\rm C}) \begin{array}{c}\includegraphics[scale=0.18]{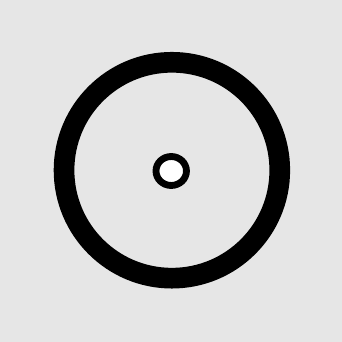}\end{array}=(q+q^{-1})\begin{array}{c}\includegraphics[scale=0.18]{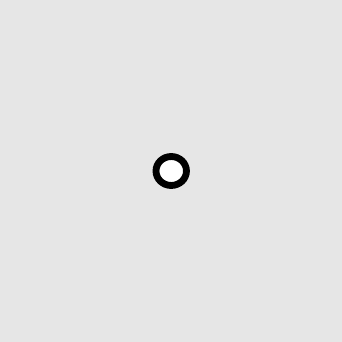}\end{array},\\
&({\rm D}) \begin{array}{c}\includegraphics[scale=0.18]{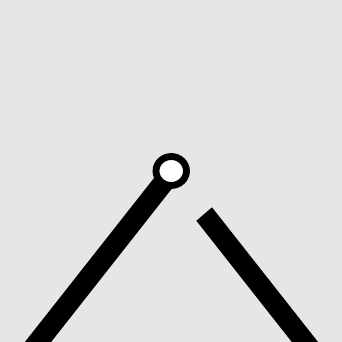}\end{array}=v^{-1}\Big(q^{1/2}\begin{array}{c}\includegraphics[scale=0.18]{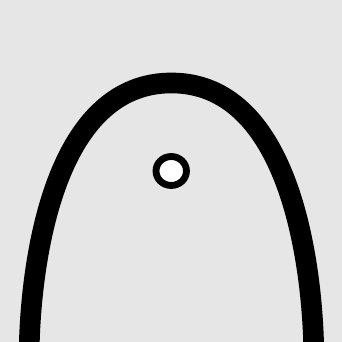}\end{array}+q^{-1/2}\begin{array}{c}\includegraphics[scale=0.18]{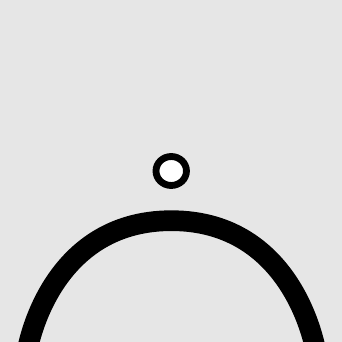}\end{array}\Big)\quad
\text{around an interior puncture $v$,}\\
&({\rm E})\ q^{-1/2}\begin{array}{c}\includegraphics[scale=0.15]{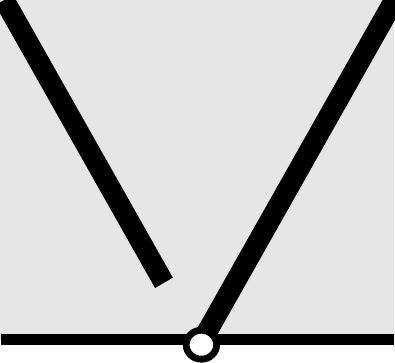}\end{array}
=q^{1/2}\begin{array}{c}\includegraphics[scale=0.15]{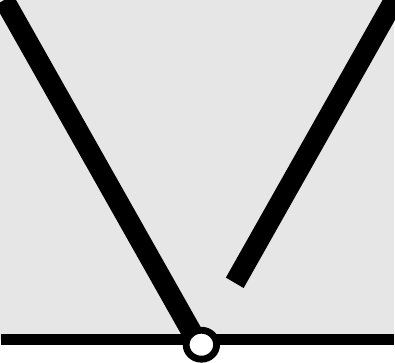}\end{array},
\\
&({\rm F})\ \begin{array}{c}\includegraphics[scale=0.15]{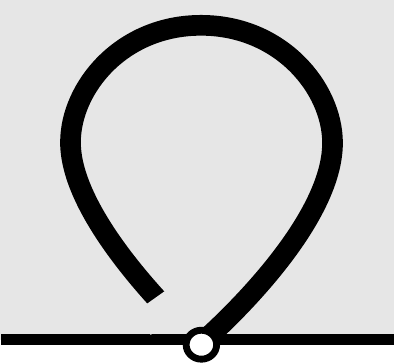}\end{array}=0=\begin{array}{c}\includegraphics[scale=0.15]{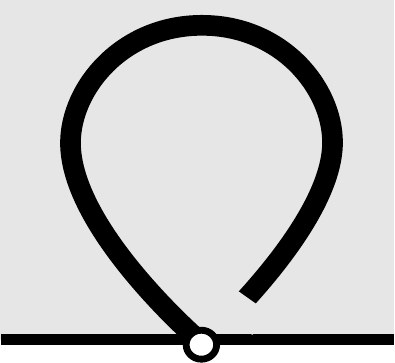}\end{array}.
\end{align*}
with a multiplication defined by stacking with respect to $(-1,1)$. 
\end{definition}

In each local relation (A)--(F), the shaded parts stand for the same region of $\Sigma$. The horizontal segments in (E) and (F), are parts of $\partial \Sigma$, and the white vertices in the pictures denote punctures or boundary marked points. In each relation, the curves except the horizontal or vertical segments are parts of $V$-tangle diagrams, which are assumed to be the same outside the shaded region. 

\begin{definition}
Let $\cS_q^\rM(\Sigma)$ be a $\CC$-subalgebra of $\cS_q^\MRY(\Sigma)$ generated by isotopy classes of $V$-tangles whose components  meet only the marked points of  $V_{\partial}$ (and do not meet any interior puncture $V_{\circ}$). 
\end{definition}

It is well known that any $V$-tangle diagram can recover a $V$-tangle and its isotopy class does not change under Reidemeister moves drawn in Figure \ref{pic:Reidemeister}; see \cite{Kau87, RY14, Mul16, BKL24}.

\begin{figure}[ht]\centering\includegraphics[width=350pt]{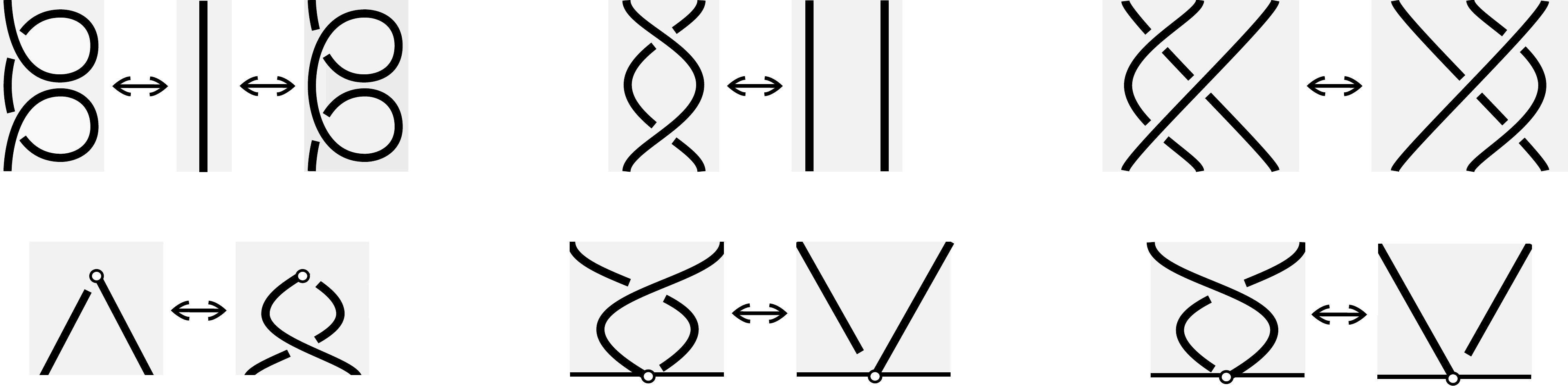}\caption{Reidemeister moves  }\label{pic:Reidemeister}\end{figure}

A \emph{state} of a $\partial$-tangle $\alpha$ is a map $s\colon \partial \alpha\cap (V_{\partial}\times (-1,1)) \to \{+,-\}$ and a \emph{stated $\partial$-tangle} is a pair $(\alpha,s)$. For simplicity, we abbreviate it to $\alpha$ if there is no confusion. Similarly, a \emph{stated $\partial$-tangle diagram} is a $\partial$-tangle diagram $\gamma$ equipped with a state $s\colon \partial \gamma\to \{+,-\}$. 

\begin{definition}[The L\^e--Roger--Yang skein algebra \cite{BKL24}]
The \emph{L\^e--Roger--Yang skein algebra} of a marked surface $(\Sigma,V)$, denoted by $\cS_{q}^{\LRY}(\Sigma)$, is the $\CC[v_{i}^{\pm}]$-algebra generated by all ambient isotopy classes of stated tangles in $\Sigma\times (-1,1)$, subject to the relations (A)--(D), (E') and (F') with a multiplication defined by stacking with respect to $(-1,1)$. 
\begin{align*}
&({\rm E}') \begin{array}{c}\includegraphics[scale=0.15]{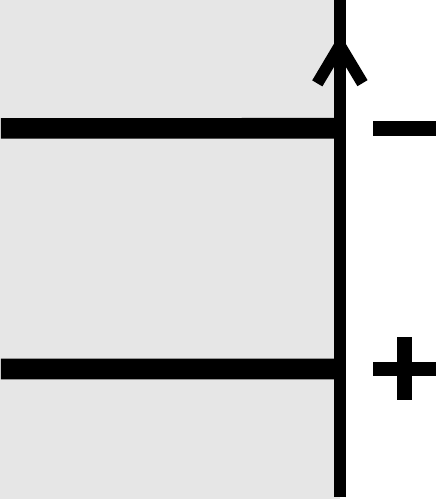}\end{array}=q^2\begin{array}{c}\includegraphics[scale=0.15]{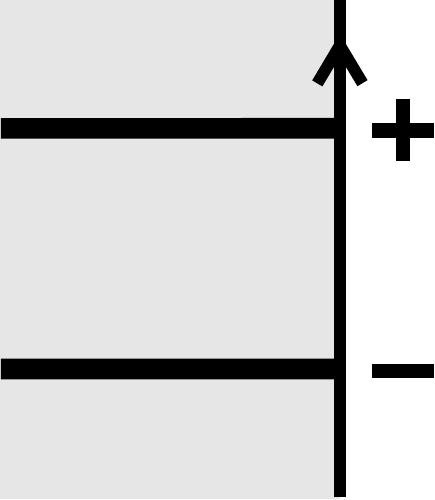}\end{array}+q^{-1/2}\begin{array}{c}\includegraphics[scale=0.15]{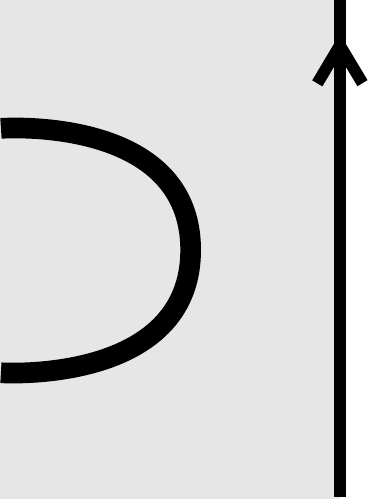}\end{array},\\
&({\rm F}')\ \begin{array}{c}\includegraphics[scale=0.15]{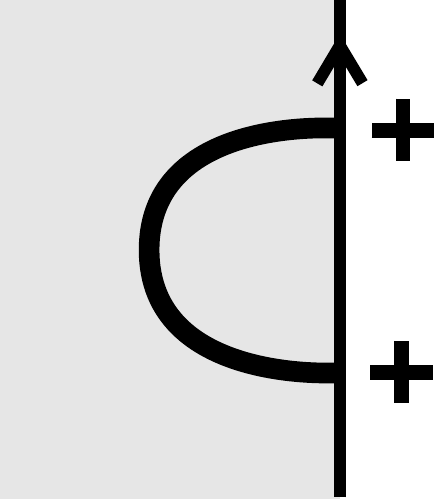}\end{array}=\begin{array}{c}\includegraphics[scale=0.15]{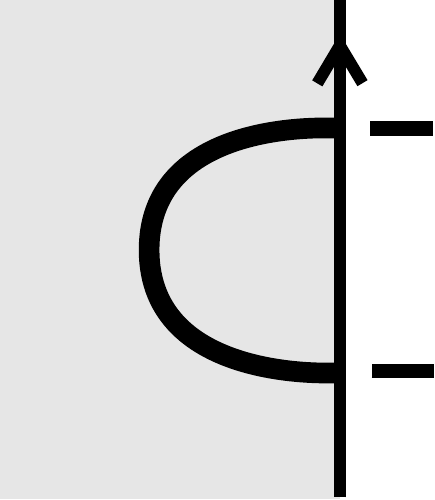}\end{array}=0 \qquad \mbox{ and } 
\qquad \begin{array}{c}\includegraphics[scale=0.15]{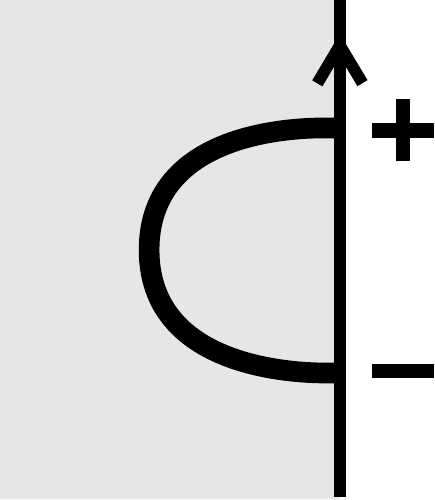}\end{array}=q^{-1/2}\begin{array}{c}\includegraphics[scale=0.15]{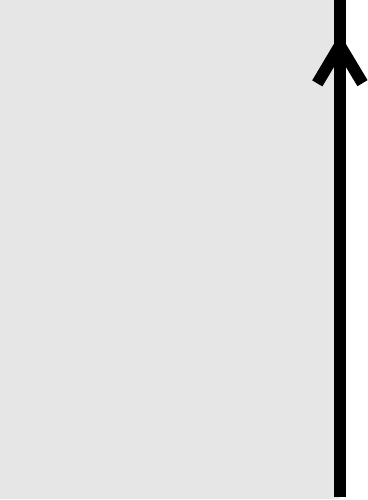}\end{array}
\end{align*}
\end{definition}

\begin{definition}
Let $\cS_q^\rL(\Sigma)$ be a $\CC$-subalgebra of $\cS_{q}^{\LRY}(\Sigma)$ generated by isotopy classes of $\partial$-tangles whose components do not end at any interior puncture $v_{i} \in V_{\circ}$. 
\end{definition}

 While similar, the algebras $\cS_q^\rM(\Sigma)$ and $\cS_q^\rL(\Sigma)$
are different than those introduced in \cite{Mul16} and \cite{Le18}.  In order to explain how, 
consider the $\CC$-algebra $\cS_{q}^{\rM+}(\Sigma)$  generated by all ambient isotopy classes of $V$-tangles in $\Sigma\times (-1,1)$ without components incident to interior punctures, subject to the all relations except (C) with a multiplication defined by stacking with respect to $(-1,1)$.  We similarly have $\cS_{q}^{\rL+}(\Sigma)$  for  stated $\partial$-tangles.
Note that the algebra $\cS_{q}^{\rM+}(\Sigma)$ was originally defined in \cite{Mul16} and is called the \emph{Muller skein algebra}.   
There is a natural epimorphism 
$ \pi: \cS_{q}^{\rM+}(\Sigma) \to \cS_{q}^{\rM}(\Sigma)$, 
so that the kernel is generated by $\ell_{v} - (q+q^{-1})$, where $\ell_{v}$ is the peripheral loop around $v$ that appears in the left diagram in the relation (C).  The algebra $\cS_{q}^{\rL+}(\Sigma)$ was originally defined in \cite{Le18}, called the \emph{stated skein algebra}, and there is a similar map  $ \pi: \cS_{q}^{\rL+}(\Sigma) \to \cS_{q}^{\rL}(\Sigma)$ as well.

\begin{remark}\label{rem:RY}
If $\Sigma$ has no interior punctures, then $\cS_{q}^{\MRY}(\Sigma) = \cS_q^\rM(\Sigma) = \cS_q^{\rM+}(\Sigma)$ agrees with the definition of the Muller skein algebra, which was originally introduced in \cite{Mul16}. 
When $\Sigma$ has no boundary, then  $\cS_{q}^{\MRY}(\Sigma)$ is specialized to the Roger-Yang skein algebra $\cS_{q}^{\mathrm{RY}}(\Sigma)$ introduced in \cite{RY14}.
\end{remark}

Recall that, in $\cS_{q}^{\MRY}(\Sigma)$, the height of the endpoints of a $V$-tangle diagram at each boundary puncture are distinct. Now we use the orientation of the boundary edges induced from that of $\Sigma$ to  move the endpoints along the boundary edge in that direction so that a higher end goes farther. As a result, we obtain what we call a positively ordered tangle diagram. We call the operation the \emph{moving trick} and it is depicted in Figure \ref{fig:movingtrick}.
\begin{figure}
$\begin{array}{c}\includegraphics[scale=0.18]{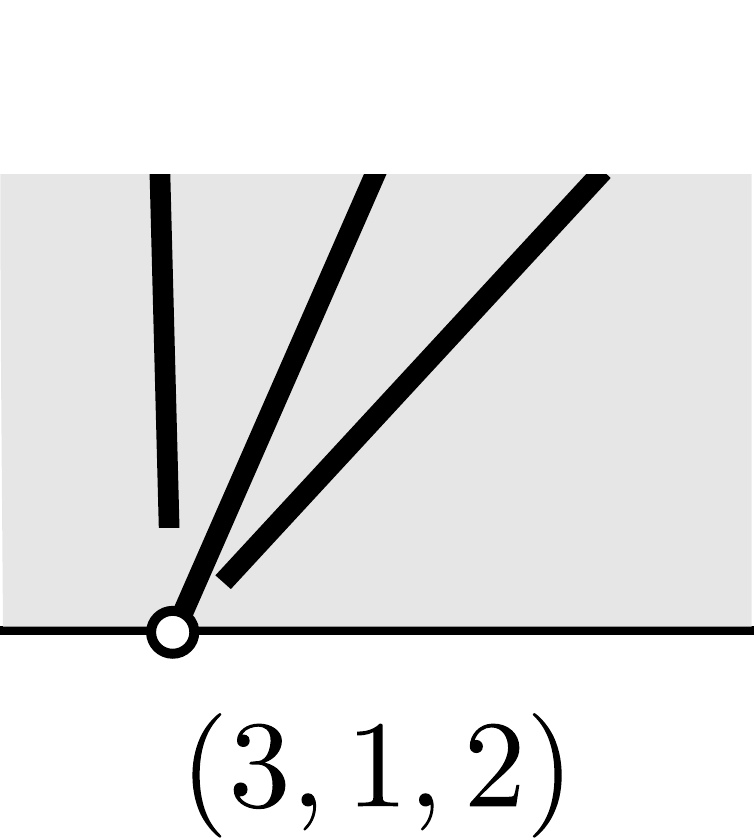}\end{array}\longleftrightarrow\begin{array}{c}\includegraphics[scale=0.18]{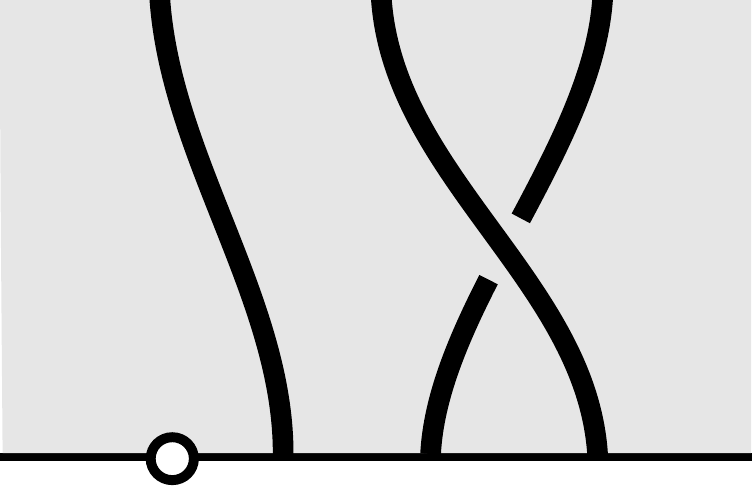}\end{array}$
\caption{An example of the moving trick. The left $V$-tangle diagram does not have the preferred crossing because the right tangle diagram has an extra crossing. The sequence under the $V$-tangle diagram denotes the height order of three strands ($1$ is the highest end.). }
\label{fig:movingtrick}
\end{figure}
We may then assign $+$ states to the endpoints of the resulting $\partial$-tangle diagram so that it can be seen as an element of $\cS_{q}^{\LRY}(\Sigma)$. As we will see later, with the bases given in Sections~\ref{ssec:LRYbases} and \ref{ssec:MRYbases}, it is clear that this operation is a well-defined injective map; see Proposition~\ref{prop:MRYvsLRY}. If we have a positively stated diagram, we may uniquely recover a corresponding $V$-tangle diagram. This moving trick map preserves the stacking structure.

We thus have shown that there is a well-defined algebra homomorphism 
\begin{equation}\label{eqn:movingmap}
	m : \cS_{q}^{\MRY}(\Sigma) \to \cS_{q}^{\LRY}(\Sigma).
\end{equation}
The morphism $m$ induces well-defined morphisms (we retain the same notation) $m : \cS_{q}^{\rM}(\Sigma) \to \cS_{q}^{\rL}(\Sigma)$ and $m : \cS_{q}^{\rM+}(\Sigma) \to \cS_{q}^{\rL+}(\Sigma)$. 

We introduce one additional generalization of skein algebra. A \emph{corner arc} is an $\partial$-tangle whose diagram consists of only one arc on $\Sigma$ encircling one boundary marked point as in Figure \ref{fig:badarc}.  A \emph{bad arc} is a corner arc with $+$ and $-$ states at two endpoints in clockwise with respect to the encircled marked point. 
\begin{figure}
\includegraphics[width=0.11\textwidth]{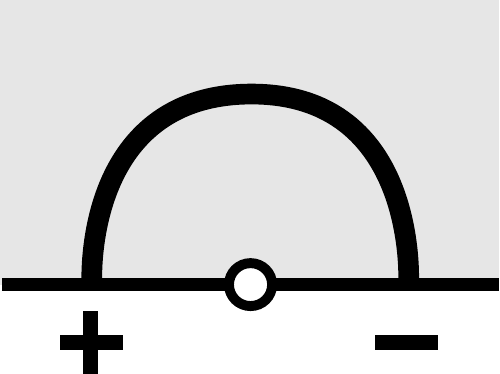}
\caption{A bad arc}
\label{fig:badarc}
\end{figure} 
Let $\overline{\cS}_{q}^{\LRY}(\Sigma)$ be the quotient algebra of $\cS_{q}^{\LRY}(\Sigma)$ modulo the ideal generated by bad arcs, and let  $p$ be its quotient map. This algebra is called the \emph{reduced LRY skein algebra}. Note that a bad arc does not meet any interior puncture, so it can be regarded as an element of $\cS_{q}^{\rL}(\Sigma)$ and $\cS_{q}^{\rL+}(\Sigma)$ as well. Thus, we may define $\overline{\cS}_{q}^{\rL}
(\Sigma)$ and $\overline{\cS}_{q}^{\rL+}(\Sigma)$ in the same way. 

A \emph{boundary arc} is an arc that is homotopic to a boundary edge relative to the set of marked points $V$. Let $\cS_{q}^{\MRY}(\Sigma)[\partial^{-1}]$ be the left Ore localization of $\cS_{q}(\Sigma)$ by the multiplicative set generated by boundary arc components. In our accompanying paper \cite[Proposition 3.1]{KMW25+}, we show that there is a natural isomorphism 
\[
	\cS_{q}^{\MRY}(\Sigma)[\partial^{-1}] \cong \overline{\cS}_{q}^{\LRY}(\Sigma).
\]
The same proof shows two more isomorphisms $\cS_{q}^{\rM}(\Sigma)[\partial^{-1}] \cong \overline{\cS}_{q}^{\rL}(\Sigma)$ and $\cS_{q}^{\rM+}(\Sigma)[\partial^{-1}] \cong \overline{\cS}_{q}^{\rL+}(\Sigma)$ which are all compatible. The last isomorphism is proved in \cite{LY22}.

In summary, we have the following commutative diagram. The maps $\pi$ and $p$ are epimorphisms, and $i$ and $m$ are monomorphisms.

\begin{equation}\label{eqn:skeinalgebradiagram}
\xymatrix{\cS_{q}^{\rM+}(\Sigma) \ar^{\pi}[r] \ar^{m}[d] &\cS_{q}^{\rM}(\Sigma) \ar^{m}[d] \ar^{i}[r] & \cS_{q}^{\MRY}(\Sigma) \ar^{m}[d]\\
\cS_{q}^{\rL+}(\Sigma) \ar^{\pi}[r] \ar^{p}[d] &\cS_{q}^{\rL}(\Sigma) \ar^{i}[r] \ar^{p}[d] & \cS_{q}^{\LRY}(\Sigma) \ar^{p}[d]\\
\overline{\cS}_{q}^{\rL+}(\Sigma) \ar^{\pi}[r] \ar@{=}[d]& \overline{\cS}_{q}^{\rL}(\Sigma) \ar^{i}[r] \ar@{=}[d] & \overline{\cS}_{q}^{\LRY}(\Sigma) \ar@{=}[d]\\
\cS_{q}^{\rM+}(\Sigma)[\partial^{-1}] \ar[r]&\cS_{q}^{\rM}(\Sigma)[\partial^{-1}] \ar[r]&\cS_{q}^{\MRY}(\Sigma)[\partial^{-1}]}
\end{equation}

\begin{remark}\label{rem:whyLRY}
Our primary interest in this paper is $\cS_{q}^{\MRY}(\Sigma)$. But the use of $\cS_{q}^{\LRY}(\Sigma)$ is useful, particularly because of a possibility of using edge/corner coordinates for each ideal triangulation $\Delta$ of $\Sigma$.  
\end{remark}

\subsection{Bases of LRY skein algebras}\label{ssec:LRYbases}
Let $\mathfrak{o}$ be an orientation of $\partial \Sigma$. 
If $\mathfrak{o}$ is the orientation induced from $\Sigma$, then $\mathfrak{o}$ is called the {\em positive order}.

A $\partial$-tangle diagram is \emph{simple} if it has neither double points in $\Sigma$ nor connected components bounding an embedded disk or a one-punctured disk or homotopic to a part of a boundary interval relative to $\partial 
\Sigma$. Given an orientation $\mathfrak{o}$ of $\partial \Sigma$, a simple tangle diagram $\alpha$ is {\em $\mathfrak{o}$-ordered} if the partial order on $\partial\alpha$ with respect to the height is increasing along each boundary interval in the direction of $\mathfrak{o}$. Note that every tangle diagram can be presented by an $\mathfrak{o}$-ordered tangle diagram by isotopy. 
In particular, if $\mathfrak{o}$ is the positive order, we call an $\mathfrak{o}$-ordered tangle diagram a positively ordered tangle diagram.

We consider an order on the set $\{+,-\}$ so that $+$ is greater than $-$. 
For a stated tangle diagram $\alpha$, 
a state $s\colon \partial \alpha\to \{+,-\}$ is {\em increasing} if $s(x)\geq s(y)$
for any $x,y\in \partial \alpha$ such that the height of $x$ is greater than that of $y$.

Let $\mathsf{B}(\Sigma,V)$ be the set of isotopy classes of increasingly stated, positively ordered simple tangle diagrams in $(\Sigma,V)$.

\begin{theorem}[{\cite[Theorem 3.6]{BKL24}}]
For a marked surface $(\Sigma,V)$, $\cS_{q}^{\LRY}(\Sigma)$ is a free $\CC[v_{i}^{\pm}]$-module generated by $\mathsf{B}(\Sigma,V)$.
\end{theorem}

\begin{remark}
One may describe a $\CC$-basis of $\cS_{q}^{\rL}(\Sigma)$ as well as that of $\cS_{q}^{\rL+}(\Sigma)$ in a similar way \cite{Le18}.
\end{remark}

\subsection{Bases of MRY skein algebras}\label{ssec:MRYbases}
Any $V$-tangle diagram has a unique over/under information at every boundary marked point such that the tangle diagram obtained by the moving trick is positively ordered and the number of crossings in $\Int \Sigma$ do not change. For a $V$-tangle diagram, we call crossings with such over/under information \emph{preferred crossings}.

\begin{definition}\label{def:RMC}
A $V$-tangle diagram is called \emph{a reduced multicurve} if 
\begin{enumerate}
\item it has no intersection in $\Int \Sigma$, including at $V_{\circ}$;
\item it has no loops which bounds an embedded open disk in $\Sigma$ including at most one interior puncture;
\item it has no null-homotopic loops and arcs;
\item at each boundary puncture, it has a preferred crossing.
\end{enumerate}
Let $\RMC$ denote the set of isotopy classes of reduced multicurves. 
\end{definition}

\begin{proposition}[\protect{\cite[Proposition 6.2]{BKL24}}]\label{prop:basis}
Let $\Sigma$ be a marked surface. Then $\cS_{q}^{\MRY}(\Sigma)$ is a free $\CC[v_{i}^{\pm}]$-module generated by $\RMC$. 
\end{proposition}

Using the above bases and the definition of the moving trick morphism $m : \cS_{q}^{\MRY}(\Sigma) \to \cS_{q}^{\LRY}(\Sigma)$, originally given in \cite{Le18} for $\cS_{q}^{\rM+}(\Sigma) \to \cS_{q}^{\rL+}(\Sigma)$, we have the following.

\begin{proposition}\label{prop:MRYvsLRY}
Let $\Sigma$ be a marked surface. Then $\cS_{q}^{\MRY}(\Sigma)$ is isomorphic to the $\CC[v_{i}^{\pm}]$-subalgebra of $\cS_{q}^{\LRY}(\Sigma)$ generated by tangles with only $+$ states.  
\end{proposition}

\subsection{Graded algebra structure}\label{ssec:gradedalgebra}
In this section, we review a natural graded $\CC$-algebra structure on $\cS_{q}^{\MRY}(\Sigma)$. 
For the set of interior punctures $V_{\circ}$, we will  assign a $\ZZ^{V_{\circ}}$-graded ring structure 
\begin{equation}\label{eqn:gradedringstructure1}
	\cS_{q}^{\MRY}(\Sigma) = \bigoplus_{\bn \in \ZZ^{V_{\circ}}}\cS_{q}^{\MRY}(\Sigma)_{\bn}
\end{equation}
on $\cS_{q}^{\MRY}(\Sigma)$, where $\cS_{q}^{\MRY}(\Sigma)_{\bn}$ is the degree $\bn \in \ZZ^{V_{\circ}}$ part. 

Recall that as a $\CC$-algebra, $\cS_{q}^{\MRY}(\Sigma)$ is generated by vertex classes  and isotopy classes of $V$-tangles. Let $\TD := \{\alpha \prod v_{i}^{m_{i}}\}$ where $\alpha \in \RMC$ and $m_{i} \in \ZZ$. Then $\TD$ is a spanning set of $\cS_{q}^{\MRY}(\Sigma)$ as a $\CC$-vector space. We first define a map 
\[
	\deg_{V} :  \TD \to \ZZ^{V_{\circ}}.
\]
Let $\be_{v}$ be the standard vector corresponding to $v \in V_{\circ}$. So any element in $\ZZ^{V_{\circ}}$ can be written uniquely as $\sum_{v \in V_{\circ}}a_{v}\be_{v}$ for some $a_{v} \in \ZZ$. For any loop $\alpha$, $\deg_{V}(\alpha) = \mathbf{0}$. If $\beta$ is an arc class connecting $v, w \in V_{\circ}$ (maybe possible that $v = w$), then $\deg_{V}(\beta) = \be_{v} + \be_{w}$. If there is a boundary marked point, we will ignore it. Finally, for the formal interior vertex class $v$ and its inverse $v^{-1}$, define $\deg_{V}(v^{\pm 1}) = \mp 2\be_{v}$. Let $\cS_{q}^{\MRY}(\Sigma)_{\bn}$ be the span of $\{x \in \TD\;|\; \deg_{V}(x) = \bn\}$. Since $\TD$ is a spanning set, we obtain \eqref{eqn:gradedringstructure1}. One can check $\deg_{V}(xy) = \deg_{V}(x) + \deg_{V}(y)$. It is also a routine calculation to check that all of the relations for the definition of $\cS_{q}^{\MRY}(\Sigma)$ are homogeneous. Therefore, we have a well-defined map 
\[
	\cS_{q}^{\MRY}(\Sigma)_{\bn} \otimes \cS_{q}^{\MRY}(\Sigma)_{\bm} \to 
	\cS_{q}^{\MRY}(\Sigma)_{\bn + \bm}
\]
and $\cS_{q}^{\MRY}(\Sigma)$ has a $\ZZ^{V_{\circ}}$-graded ring structure, as desired.

\section{Central elements}\label{sec:central}
In this section, we show that each element listed in the statement of Theorem \ref{thm:mainthm} is central, that is, commutative with all elements in $\cS_{q}^{\MRY}(\Sigma)$. Note that, by definition, all vertex classes and their formal inverses are in the center.

Define $\alpha^{(k)}$ to be the multicurve consisting of $k$ copies of $\alpha$ that are parallel in the direction of the framing of $\alpha$. Since at every end point of $\alpha$, the framing is vertical. Thus, if $\alpha$ is a $V$-tangle, then $\alpha^{(k)}$ is also a well-defined $V$-tangle. For any polynomial $P = \sum_k c_k x^k$, then $\alpha^P := \sum c_k \alpha^{(k)}$ is the \emph{threading} of $P$ along $\alpha$. If $\alpha$ is represented by a $V$-tangle diagram without any self intersection, e.g.,  $\alpha \in \RMC$, then  $\alpha^{(k)} = \alpha^{k} \in \cS_{q}^{\MRY}(\Sigma)$ and $\alpha^{P} = P(\alpha)$. The definition is extended linearly to skeins in $\cS_{q}^{\MRY}(\Sigma)$.

Central elements of the ordinary skein algebra $\cS_{q}(\Sigma)$ come from a threading operation \cite{BW16} of Chebyshev polynomials. Our main task in this section is to generalize the threading appropriately to include arc classes as well. 
In \cite{BW16}, for the Kauffman bracket skein algebra $\cS_{q}(\Sigma)$, it was shown that there is an algebra monomorphism $\Phi : \cS_{1}(\Sigma) \to \cS_{q}(\Sigma)$, which is called the \emph{Chebyshev--Frobenius homomorphism}. This homomorphism extends to the one between stated skein algebras $\Phi : \cS_{1}^{\rL+}(\Sigma) \to \cS_{q}^{\rL+}(\Sigma)$ in \cite{BL22}, as we describe below. Note that in this setup, there is no arc class ending at interior punctures.

For any loop $\alpha \in \cS_{1}^{\rL+}(\Sigma)$, we define $\Phi(\alpha) :=\alpha^{T_{n}} $, where $T_{n}(x)$ is the $n$-th Chebyshev polynomial recursively defined as 
\[
	T_{0}(x) = 2, \quad T_{1}(x) = x, \quad T_{k+2}(x) = xT_{k+1}(x) - T_{k}(x). 
\]
For any arc $\alpha \in \cS_{1}^{\rL+}(\Sigma)$ connecting boundary marked points, we set $\Phi(\alpha) := \alpha^{(n)}$. 
 If we have a disjoint union of arcs and loops, we take the disjoint union of the threading of each component.  
  It induces a well-defined algebra homomorphism $\Phi : \cS_{1}^{\rL+}(\Sigma) \to \cS_{q}^{\rL+}(\Sigma)$ \cite[Theorem 2]{BL22}. Under the existence of interior punctures, this algebra homomorphism can be extended to 
\begin{equation}\label{eqn:Chebyshevmap+}
	\Phi : \CC[v_{i}^{\pm}] \otimes_{\CC} \cS_{1}^{\rL+}(\Sigma) \to \CC[v_{i}^{\pm}] \otimes_{\CC} \cS_{q}^{\rL+}(\Sigma).
\end{equation}
One may check that the morphism $\Phi$ induces a similar map 
\begin{equation}\label{eqn:Chebyshevmap}
	\Phi : \CC[v_{i}^{\pm}] \otimes_{\CC} \cS_{1}^{\rL}(\Sigma) \to \CC[v_{i}^{\pm}] \otimes_{\CC} \cS_{q}^{\rL}(\Sigma)
\end{equation}
because for any peripheral loop $\ell_{v}$, $\ell_{v}^{T_{n}} - 2$ is in the ideal generated by $\ell_{v} - (q+q^{-1})$.

Since $\cS_{1}^{\rL}(\Sigma)$ is commutative, so is its image $\mathrm{im}\; \Phi$. Moreover, based on earlier work of \cite{Le15, BW16, BL22}, it was shown that  $\mathrm{im}\; \Phi \subseteq Z(\cS_{q}^{\rL}(\Sigma))$ \cite[Corollary 4.3]{Yu23}. Since $\cS_{q}^{\MRY}(\Sigma)$ can be understood as a subalgebra of $\cS_{q}^{\LRY}(\Sigma)$ (Proposition \ref{prop:MRYvsLRY}), we obtain the following result.

\begin{lemma}\label{lem:imChebysheviscommutative}
\begin{enumerate}
\item Let $\alpha \in \cS_{q}^{\MRY}(\Sigma)$ be a loop class on $\Sigma$. Then $\alpha^{T_{n}}$ is central. In particular, if $\alpha$ does not have any self-intersection, $T_{n}(\alpha)$ is central. 
\item Let $\beta \in \cS_{q}^{\MRY}(\Sigma)$ be an arc class on $\Sigma$ connecting two boundary marked points. Then $\beta^{(n)}$ is central. In particular, if $\beta$ does not have any self-intersection, $\beta^{n}$ is central. 
\end{enumerate}
\end{lemma}

\begin{proof}
For a loop $\alpha \in \cS_{q}^{\MRY}(\Sigma)$, $m(\alpha^{T_{n}}) = \alpha^{T_{n}} \in \cS_{q}^{\LRY}(\Sigma)$. And $\alpha^{T_{n}}$ is also in $\cS_{q}^{\rL}(\Sigma)$ and it is central there. So it remains to check that it is commutative with arc classes joining an interior puncture and another marked point (either interior puncture or boundary marked point). Since $\alpha^{T_{n}}$ does not intersect any interior punctures, we may check locally its transparency with every arc class following the argument in \cite[Corollary 2.3]{Le15}.  
Hence $\alpha^{T_{n}} \in Z(\cS_{q}^{\LRY}(\Sigma))$. Since $m$ is a monomorphism, we may conclude that $\alpha^{T_{n}} \in Z(\cS_{q}^{\MRY}(\Sigma))$. This proves (1). For part (2), this follows from \cite{BL22}. \end{proof}

We also obtain the following central element. 

\begin{lemma}\label{lem:boudnaryarcs}
For each component $D$ of $\partial \Sigma$ with boundary arcs $\beta_{1}, \beta_{2}, \cdots, \beta_{k}$ in any order, the product $ \prod_{i}\beta_{i}$ is central in $\cS_{q}^{\MRY}(\Sigma)$.
\end{lemma}

\begin{proof}
For simplicity,   consider first the case where $\beta_{1}, \beta_{2}, \cdots, \beta_{k}$  is in cyclic order.  Set $\beta_{k+1} := \beta_{1}$, and let $\beta_{D} := \prod_{i}\beta_{i}$ in that case. 
The boundary class is disjoint with all loop classes, it is sufficient to check the commutativity with an arc $\gamma$ whose end is a boundary marked point $v$ on $D$, between $\beta_{i}$ and $\beta_{i+1}$. Using the height exchange relation (E) in Definition \ref{def:MRY} twice, we can exchange the order of $\gamma$ and $\beta_{i}\beta_{i+1}$. During this process we need to multiply $qq^{-1}$, so $\gamma\beta_{i}\beta_{i+1} = \beta_{i+1}\beta_{i}\gamma$. If both ends of $\gamma$ meet $D$, we may apply the argument twice.

For the more general case, note that the order of the product $\prod_{i}\beta_{i}$ does not really matter indeed. For any permutation $\sigma \in S_{k}$ and a product $\prod_{i}\beta_{\sigma(i)}$, applying the relation (E) in Definition \ref{def:MRY}, we can conclude that $\prod_{i}\beta_{\sigma(i)}$ is equal to $q^{t}\prod_{i}\beta_{i} = q^{t}\beta_{D}$ for some $t \in \ZZ$, hence it is central, too. 
\end{proof}

Combining Item (2) of Lemma \ref{lem:imChebysheviscommutative} and Lemma \ref{lem:boudnaryarcs}, we obtain the following result. We fix a boundary component $D$ of $\partial \Sigma$, where $D$ has $k$ marked points. Let  $\beta := \prod_{i}\beta_{i}^{m_{i}}$ be a product of boundary arcs on $D$ with multiplicity $\deg_{D}(\beta) := (m_{i}) \in \ZZ^{D}$.  

\begin{corollary}\label{cor:boundarycenterdegree}
The product of boundary arcs $\beta$ is in $Z(\cS_{q}^{\MRY}(\Sigma))$ if and only if $\deg_{D}(\beta) = (m_{i})$ is in the sublattice $M_{D} = \langle  n\be_{1}, n\be_{2} , \ldots,  n\be_{k} ,  \sum_{i=1}^{k}\be_{i} \rangle \subset \ZZ^{D}$.  \end{corollary}

\begin{proof}  
We know that $\beta_{D} = \prod_{i}\beta_{i}$  and $\beta_{i}^{n}$ are central, and 
$\deg_{D}(\beta_{D}) = \sum_{i}\be_{i}$ and $\deg_{D}(\beta_{i}^{n}) = n\be_{i}$. For any $\beta$ with $\deg_{D}(\beta) \in M_{D}$, up to a scalar multiple by $q^{t}$, $\beta$ is a product of $\beta_{D}$ and $\beta_{i}^{n}$, hence $\beta \in Z(\cS_{q}^{\MRY}(\Sigma))$.

Conversely, suppose that $\deg_{D}(\beta) = \sum m_{i}\be_{i} \notin M_{D}$. Eliminating $\beta_{D}$ and $\beta_{i}^{n}$ if necessary, we may assume that $(m_{i})$ is nontrivial, $0 \le m_{i} < n$ and there is one coefficient (say $m_{1}$) which is zero. We may explicitly compute that $\beta \beta_{1} = q^{m_{2}}\beta_{1}\beta$. Since $0 < m_{2} < n$, $q^{m_{2}} \ne 1$, and since $\cS_{q}^{\MRY}(\Sigma)$ is a domain \cite[Theorem 5]{BKL24}, $\beta$ is not central. 
\end{proof}

We now turn our attention to new arc classes that connect some interior punctures. 

\begin{lemma}\label{lem:arcboundaryinterior}
Let $\beta$ be an arc class connecting a boundary puncture $b$ and an interior puncture $v$. Suppose $\beta$ does not admit any self-intersection on its diagram. Then $\beta^{n}$ is central. 
\end{lemma}

\begin{proof}
Using (D), (E), and Reidemeister moves, 
\[
	\beta^{n}=\begin{cases}
v^{-(n-1)/2}q^{n(n-1)/4}\begin{array}{c}\includegraphics[scale=0.18]{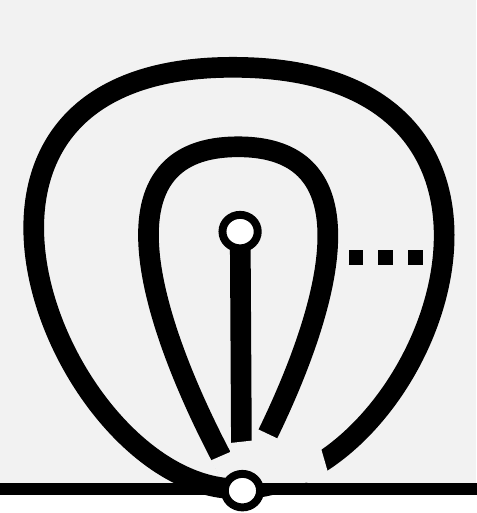}\end{array}& \text{if $n$ is odd},\\
v^{-n/2}q^{n(n-1)/4}\begin{array}{c}\includegraphics[scale=0.18]{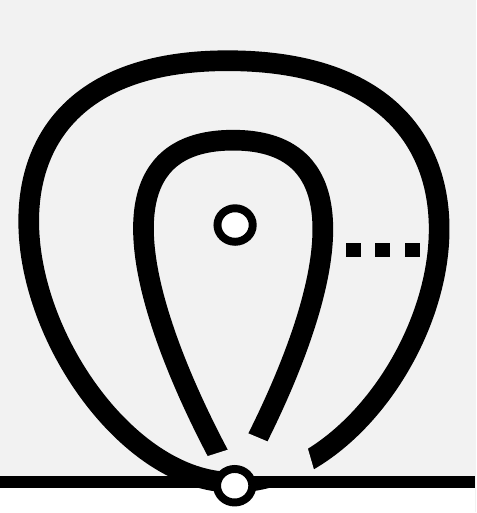}\end{array}& \text{if $n$ is even}, 
\end{cases}
\]
where the heights of endpoints at $b$ increase in counter-clockwise. 

Let $[\beta^{n}]$ denote $\begin{array}{c}\includegraphics[scale=0.18]{draws/arcs_odd.pdf}\end{array}$ or $\begin{array}{c}\includegraphics[scale=0.18]{draws/arcs_even.pdf}\end{array}$ depending on the parity of $n$. It suffices to show that $[\beta^{n}]$ is central to show the claim. 

For any $V$-tangle diagram $\gamma$, we show that $[\beta^{n}]\gamma = \gamma [\beta^{n}]$. First we consider 3 cases where $\gamma$ meets $\beta$: 
\begin{enumerate}
    \item $\gamma$ intersects with $\beta$ just once at $b$, 
    \item $\gamma$ intersects with $\beta$ just once only at $v$, 
    \item $\gamma$ intersects with $\beta$ just once in  $\Int \Sigma\setminus V$. 
\end{enumerate}

In the first case, we have $\gamma[\beta^{n}] =q^{-n}[\beta^{n}]\gamma = [\beta^{n}]\gamma$ using (E) and Reidemeister moves. 

Next, we consider the second case. When $n=3$, 
\begin{eqnarray*}
v\gamma[\beta^{3}]&=&v\begin{array}{c}\includegraphics[scale=0.18]{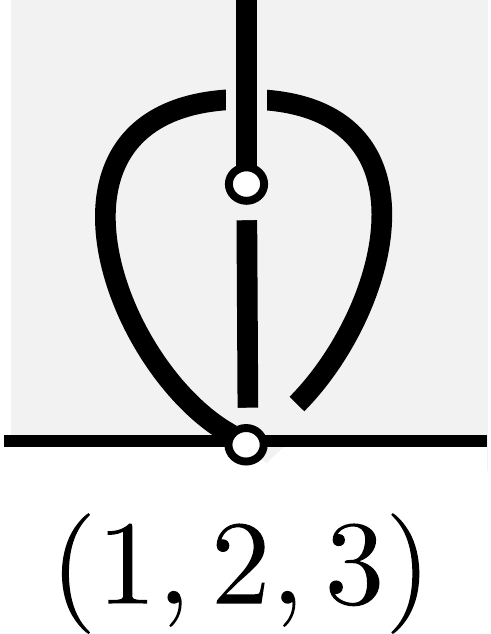}\end{array}=q^{1/2}\begin{array}{c}\includegraphics[scale=0.18]{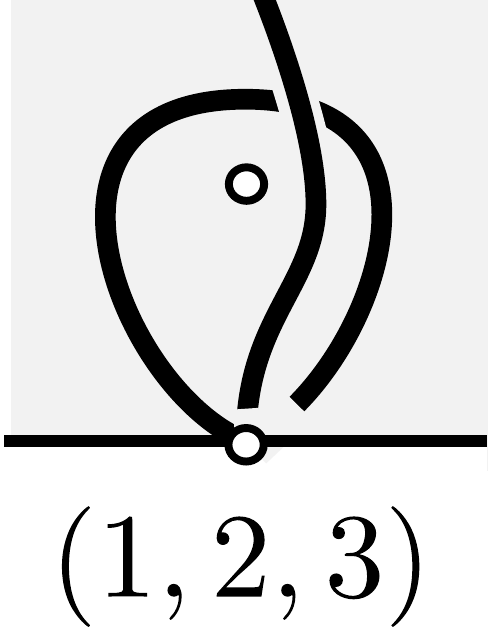}\end{array}+q^{-1/2}
\begin{array}{c}\includegraphics[scale=0.18]{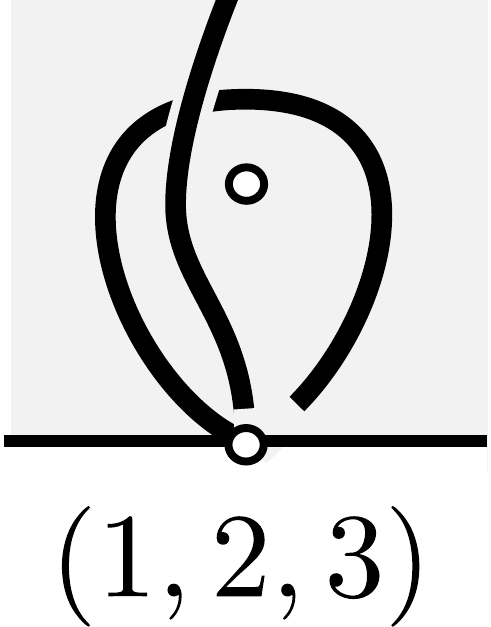}\end{array}
\\
&=&q^{1/2}\begin{array}{c}\includegraphics[scale=0.18]{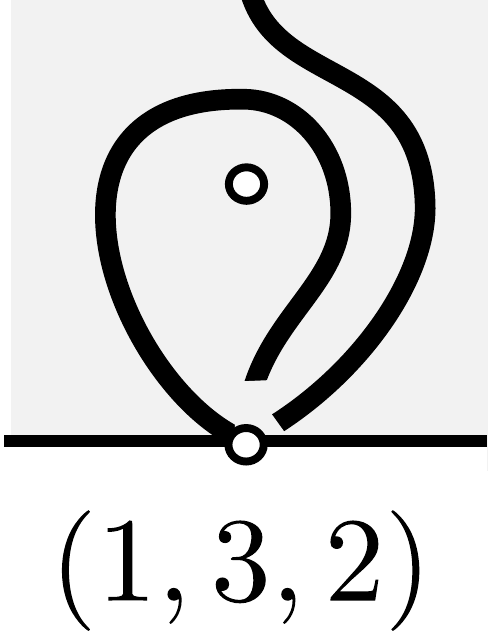}\end{array}+q^{-3/2}
\begin{array}{c}\includegraphics[scale=0.18]{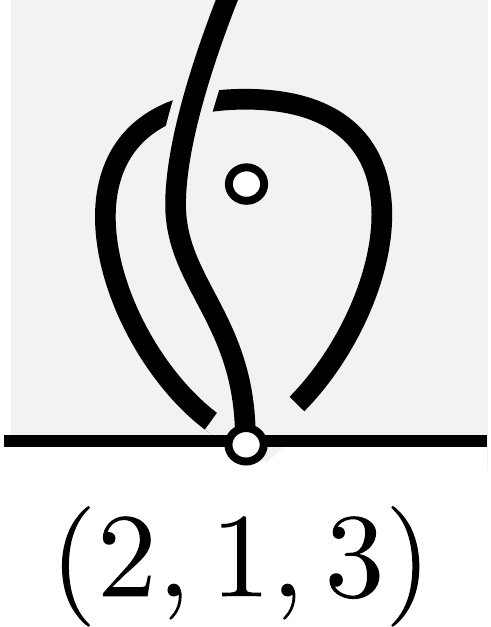}\end{array}
=q^{3/2}\begin{array}{c}\includegraphics[scale=0.18]{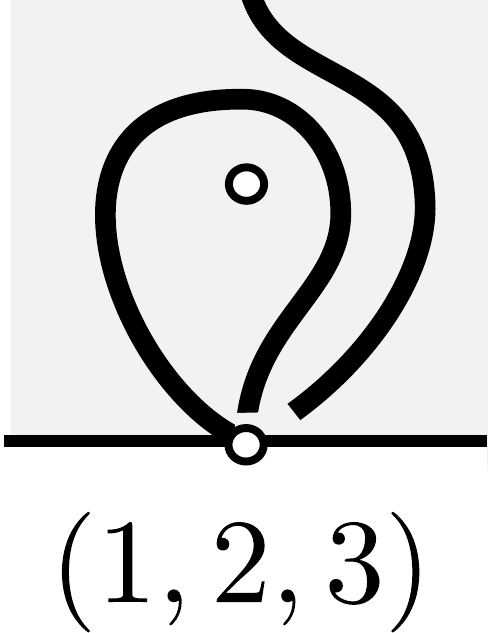}\end{array}+q^{-3/2}
\begin{array}{c}\includegraphics[scale=0.18]{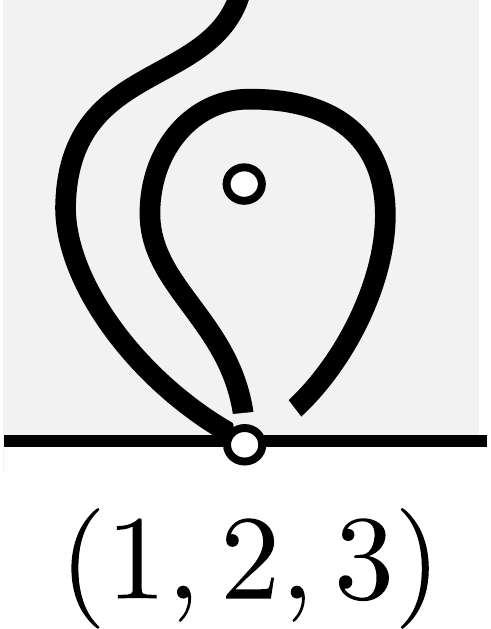}\end{array}\\
&=&q^{3/2}[\beta^{2}]R+q^{-3/2}L[\beta^{2}],
\end{eqnarray*}
where the numbers below each picture denote the height information of the three strands incident to $b$. A smaller number implies the corresponding strand is higher. We denote $R$ and $L$ by the middle and right (partial) arcs in Figure \ref{pic:gamma_rl} respectively, and the equalities follows from (D), (E) and Reidemeister moves. By the same argument, we also have $v[\beta^{3}] \gamma =q^{-3/2}[\beta^{2}]R+q^{3/2}L[\beta^{2}]$. Since $q$ is a root of unity of order $n$, we have $[\beta^{3}]\gamma=\gamma[\beta^{3}]$. 

For a general $n$, it is a routine computation to check 
\[
	[\beta^{n}]\gamma=q^{n/2}[\beta^{n-1}]R+q^{-n/2}L[\beta^{n-1}]=
q^{-n/2}[\beta^{n-1}]R+q^{n/2}L[\beta^{n-1}]=\gamma[\beta^{n}]
\]
when $q$ is a root of unity of order $n$. 

\begin{figure}[ht]\centering\includegraphics[width=140pt]{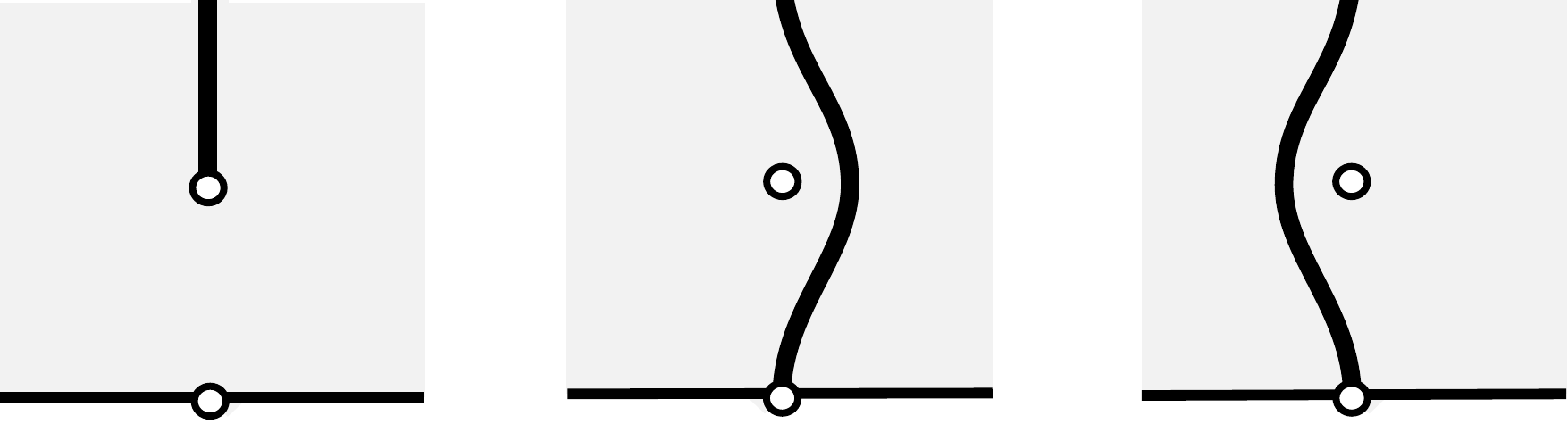}\caption{Left: (partial) picture of $\gamma$,\quad Middle: (partial) picture of $R$,\quad Right: (partial) picture of $L$, Outside the shaded part, the three arcs are the same.}\label{pic:gamma_rl}\end{figure}

In the third case, using (D) and Reidemeister moves, 
\begin{eqnarray*}
\begin{array}{c}\includegraphics[scale=0.18]{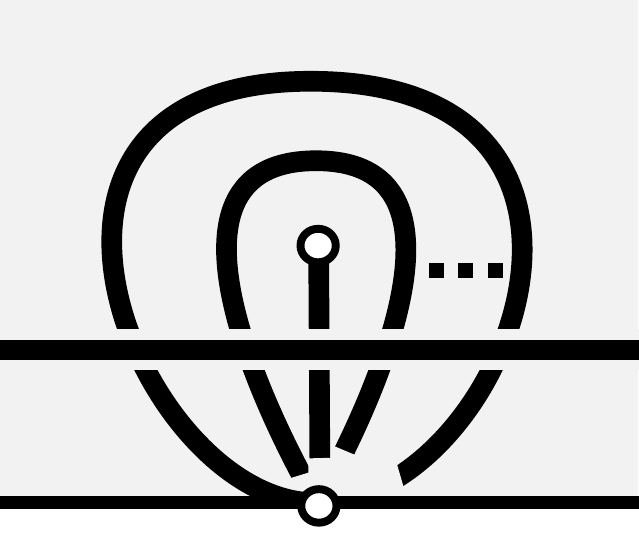}\end{array}&=&vq^{1/2}\begin{array}{c}\includegraphics[scale=0.18]{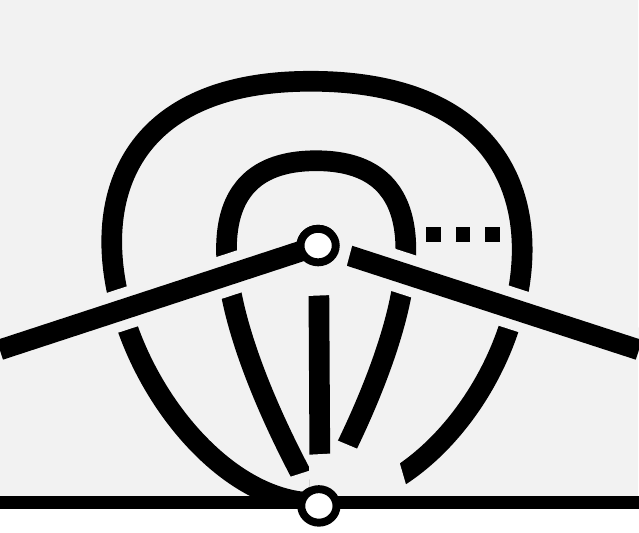}\end{array}-q\begin{array}{c}\includegraphics[scale=0.18]{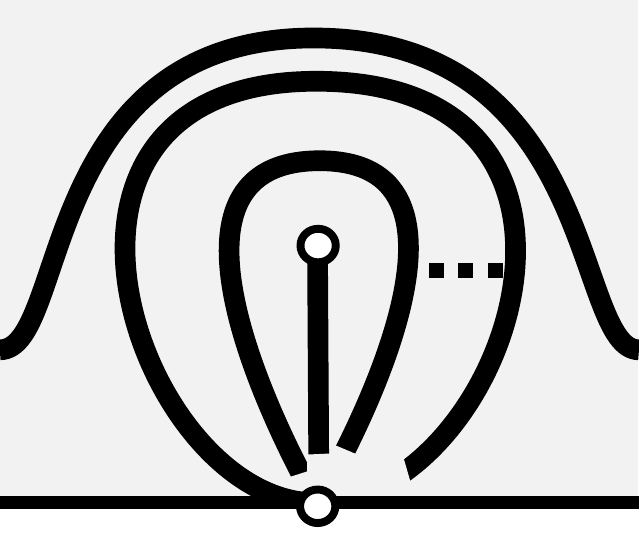}\end{array}\\
&=&vq^{1/2}\begin{array}{c}\includegraphics[scale=0.18]{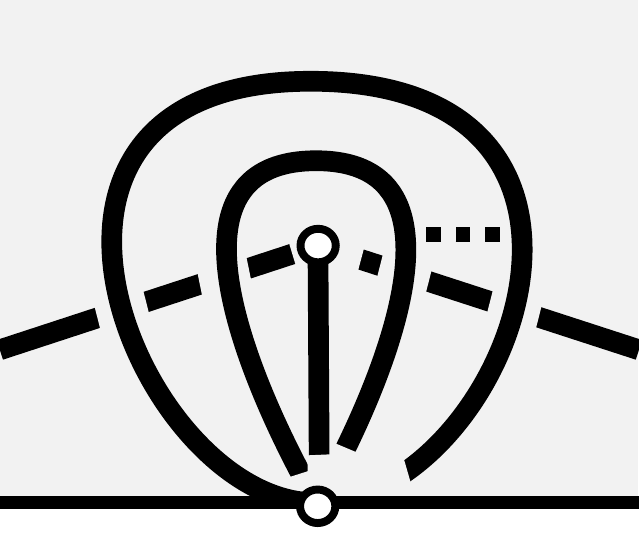}\end{array}-q\begin{array}{c}\includegraphics[scale=0.18]{draws/prod_arcs_odd_ho.pdf}\end{array}
=\begin{array}{c}\includegraphics[scale=0.18]{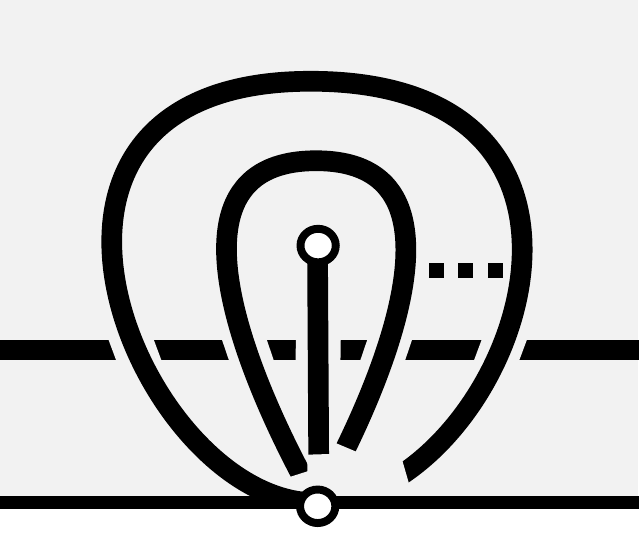}\end{array},
\end{eqnarray*}
where the heights of endpoints at the boundary marked point increase in counter-clockwise and the second equality holds from the second case. 

The above three cases show $[\beta^{n}]$ is `transparent.' Thus, for any  diagram $\gamma$, $[\beta^{n}]$ commutes with $\gamma$. Thus, we may conclude that $[\beta^{n}]$ is a central element. 
\end{proof}

Finally, we show the centrality of an element coming from an arc class connecting punctures.  For the degree $n$ Chebyshev polynomial $T_{n}(x)$, consider the formal expression 
\begin{equation*}
	\frac{1}{\sqrt{v}\sqrt{w}}T_{n}(\sqrt{v}\sqrt{w} x). 
\end{equation*}

\begin{remark}\label{rmk:Tnbetaisinthealgegbra}
Let $\beta$ be an arc class connecting two distinct interior punctures $v$ and $w$.
Even though $\cS_{q}^{\MRY}(\Sigma)$ does include the square roots of $v$ and $w$,  the arc class $\beta$ threaded by $\frac{1}{\sqrt{v}\sqrt{w}}T_{n}(\sqrt{v}\sqrt{w} x)$  remains an element of $\cS_{q}^{\MRY}(\Sigma)$. This is because $T_{n}(x) $ has only odd degree terms, thus every power of $v$ and $w$ are integers, not half integers. In particular, one can check that 
\[
	\frac{1}{\sqrt{v}\sqrt{w}}T_{3}(\sqrt{v}\sqrt{w} x) = vw x^{3} - 3 x, \quad
	\frac{1}{\sqrt{v}\sqrt{w}}T_{5}(\sqrt{v}\sqrt{w}x ) = v^{2}w^{2} x^{5} - 5vwx ^{3} + 5 x,\quad  \cdots
\]
So, for example,  $\beta$ threaded by $\frac{1}{\sqrt{v}\sqrt{w}}T_{3}(\sqrt{v}\sqrt{w} x)$ is the skein $vw \beta^{(3)} - 3  \beta$.  
\end{remark}

\begin{remark} When $\beta$ whose diagram does not admit any self-intersections, then $\beta$ threaded by $\frac{1}{\sqrt{v}\sqrt{w}}T_{n}(\sqrt{v}\sqrt{w} x)$ is identical to the skein $\frac{1}{\sqrt{v}\sqrt{w}}T_{n}(\sqrt{v}\sqrt{w} \beta)$.  However, this is not in general true for diagrams with self-intersection. 
\end{remark}

\begin{remark}\label{rmk:homogenization}
We may also understand the element  $\frac{1}{\sqrt{v}\sqrt{w}}T_{n}(\sqrt{v}\sqrt{w} \beta)$ as a homogenization of $T_{n}(\beta)$ with respect to the $\ZZ^{V_{\circ}}$-grading. 
\end{remark}

\begin{lemma}\label{lem:interiorarcs}
Let $\beta$ be an arc class connecting two distinct interior punctures $v$ and $w$, without any self-intersection on its diagram. Then $\beta$ threaded by the polynomial $\frac{1}{\sqrt{v}\sqrt{w}}T_{n}(\sqrt{v}\sqrt{w} x)$ is central. 
\end{lemma}

\begin{proof}
The centrality was first observed in \cite[Corollary 5.3]{Kar24}. In \cite{Kar24}, it was calculated in a degenerated version of the skein algebra that one can obtain by setting $v = 1$ for all vertices. To make the same proof work in $\cS_{q}^{\MRY}(\Sigma)$, we need to allow a resolution using the puncture-skein relation whenever there are two arc classes meeting at the same interior puncture. Multiplying $\sqrt{v}\sqrt{w}$ on $\beta$ enables us that such a resolution is possible, because whenever we have two arc components meeting at $v$, we will have $(\sqrt{v})^{2} = v$. Using exactly the same computation with \cite[Section 5]{Kar24}, we obtain the centrality. 
\end{proof}

\begin{remark}
We do not need to consider the case that $\beta$ is an arc whose two ends are the same interior puncture $v$. This is because after multiplying by  $v$, then $v\beta$ is the sum of two loops, hence it is in $\mathrm{im}\; \Phi$. 
\end{remark}

\begin{remark}
We observed that the Chebyshev-Frobenius morphism $\Phi : \cS_{1}^{\rL}(\Sigma) \to \cS_{q}^{\rL}(\Sigma)$ does not extend to $\Phi : \cS_{1}^{\MRY}(\Sigma) \to \cS_{q}^{\MRY}(\Sigma)$ as an algebra homomorphism. Using the basis $\RMC$, we obtain an extended linear map $\Phi : \cS_{1}^{\MRY}(\Sigma) \to \cS_{q}^{\MRY}(\Sigma)$, but it does not preserve the skein relations. 
\end{remark}


\section{Characterization of the center}\label{sec:centercomputation}

In this section, we prove Theorem \ref{thm:mainthm}.    
The outline of the proof is as follows: In  Section \ref{sec:central}, we showed that the elements in the statement of Theorem \ref{thm:mainthm} are central. To show that they generate the center, we will follow a strategy adopted by Frohman, Kania-Bartoszynska, L\^e \cite{FKBL19}, Korinman \cite{Kor21}, and Yu \cite{Yu23}. Suppose that we have a total term order on the set $\RMC$ of reduced multicurves, which behaves nicely with respect to the multiplication. For any central element $z$, we show that the leading term of $z$ is equal to the leading term of an element $z'$ that can be described in terms of the elements in Theorem \ref{thm:mainthm}. Then we may apply the induction. To realize this strategy, in Section \ref{ssec:coordinates}, we introduce a generalized version of the well-known edge coordinates allowing us to define the coordinates for ideal arcs as well.

\subsection{Generalized edge coordinates for $\mathsf{B}(\Sigma,V)$ and $\RMC$}\label{ssec:coordinates}
We fix an ideal triangulation $\Delta$ on $\Sigma$, where $\Delta$ may admit self-glued triangles. Recall that $V = V_{\circ} \sqcup V_{\partial}$ is the set of vertices in $\Delta$ and $E$ is the set of edges in $\Delta$. We say that an edge $e$ is a boundary edge if $e \subseteq \partial \Sigma$ and is an interior edge otherwise. 
By a corner, we mean a pair $(T, \{e_{1}, e_{2}\})$ where $T$ is a triangle in $\Delta$ and $e_{1}, e_{2}$ are two edges of $T$, and let the set of all corners be denoted by $C$.

Recall that $\cS_{q}^{\LRY}(\Sigma)$ has basis $\mathsf{B}(\Sigma,V)$ consisting of increasingly stated, positively ordered simple tangles. Up to isotopy, we may assume that a simple tangle diagram $\alpha \in \mathsf{B}(\Sigma,V)$ is in normal position with respect to $\Delta$, and there is no turning back over each edge $e \in E$. Then the \emph{edge coordinate} $f(\alpha) \in \ZZ^{E}$ is defined by $f(\alpha)(e)= |\alpha \cap e|$, the intersection number between $\alpha$ and $e$. Since any diagram $\alpha$ does not intersect any vertices in $V$, it also admits a well-defined \emph{corner coordinate} $g(\alpha) \in \ZZ^{C}$ that counts the number of components around each corner $c \in C$. These two coordinate systems are equivalent -- there is an explicit coordinate change formula. See \cite[Section 3.2]{Mat07}. Once one imposes a total order on $E$, we may induce a lexicographical order on $\ZZ^{E}$. This provides an order on $\mathsf{B}(\Sigma, V)$. This order is nearly a total order -- the only non-distinguishable elements are the same topological simple tangle diagrams with two different states.

Now, we move to $\cS_{q}^{\MRY}(\Sigma)$. For $\alpha \in \RMC$, we define edge coordinates $f(\alpha) \in \frac{1}{2}\ZZ^E$ as follows, using the square trick. We fix a term order on $E$, and hence on $\ZZ^{E}$. Suppose that $\alpha \in \RMC$ does not intersect any interior puncture, hence $\deg_{V}(\alpha) = \mathbf{0}$. Then applying the moving trick, we may identify $\cS_{q}^{\MRY}(\Sigma)$ with a subalgebra of $\cS_{q}^{\LRY}(\Sigma)$ (Proposition \ref{prop:MRYvsLRY}). In particular, $\alpha$ can be identified with a curve in $\cS_{q}^{\LRY}(\Sigma)$ without any intersection with interior punctures or boundary marked points with $+$ states. Thus, we obtain well-defined edge coordinates.

If $\alpha \in \RMC$ but $\deg_{V}(\alpha) \ne \mathbf{0}$, then because at most one  $\alpha \in \RMC$ ends at any interior bundary puncture, $\deg_{V}(\alpha) \in \sum_{v \in W}\be_{v}$ for some subset $W \subset V_{\circ}$. So if we take $\alpha^{2}\prod_{v \in W}v$, it is a linear combination of reduced multicurves with $\deg_{V} = \mathbf{0}$.  Thus for each term, the edge coordinate is well-defined. We denote $\lt_{f}(\alpha^{2}\prod_{v \in W}v)$ by the leading term of $\{f(\alpha_{i})\} \in \ZZ^{E}$ where $\{\alpha_{i}\}$ is the set of reduced multicurves that appear on $\alpha^{2}\prod_{v \in W}v$. We may define the generalized edge coordinate as
\begin{equation}\label{eqn:generalizededgecoordinates}
	f(\alpha) := \frac{1}{2}\lt_{f}(\alpha^{2}\prod_{v \in W}v) \in \frac{1}{2}\ZZ^{E}.
\end{equation}

\begin{remark}\label{rmk:formula}
We have an explicit formula for $f(\alpha)$. Note that if $\alpha, \alpha' \in \RMC$ are disjoint, then $f(\alpha\alpha') = f(\alpha) + f(\alpha')$. Thus, it is sufficient to describe the formula for arc classes ending at an interior puncture. The case of a loop class is classical.

Suppose that $\alpha$ is an arc class connecting two distinct punctures $v, w \in V_{\circ}$. Let $E(v)$ be the set of edges adjacent to $v$. Then 
\[
	f(\alpha) = \begin{cases}
	\frac{1}{2}\left(\sum_{t \in E(v)}\be_{t} + \sum_{t \in E(w)}\be_{t}\right) + \sum_{p \in t \cap \alpha}\be_{t}, &\mbox{if $\alpha$ is not an edge in $\Delta$,}\\
	\frac{1}{2}\left(\sum_{t \in E(v)}\be_{t} + \sum_{t \in E(w)}\be_{t}\right) - \be_{s}, &\mbox{if $\alpha$ is an edge $s \in \Delta$.}
	\end{cases}
\]
On the other hand, suppose that $\alpha$ is an arc class connecting one puncture $v$ and a marked point $w$. Let $\underline{\alpha}$ be a topological arc connecting $v$ and a boundary component near $w$, obtained by applying the moving trick. Then 
\[
	f(\alpha) = \frac{1}{2}\sum_{t \in E(v)}\be_{t} + \sum_{p \in t \cap \underline{\alpha}}\be_{t}.
\]

We may interpret the formula like this: $f(\alpha)$ is obtained from the `ordinary edge coordinates' $\sum_{p \in t \cap \underline{\alpha}}\be_{t}$ by modifying it with the contribution of each interior puncture. 
\end{remark}

\begin{remark}\label{rem:cornerisforsquare}
We may also define the generalized corner coordinate for $\alpha \in \RMC$ as in \cite[Section 6]{MW21}. But it does not have a simple computational formula, so we do not use it here. In the proof in the next section, we use both edge coordinates and corner coordinates. However, corner coordinates are used only for the curve classes \emph{without intersecting any marked points}. 
\end{remark}

\subsection{Proof of Theorem \ref{thm:mainthm}}\label{ssec:proofA}
We fix an ideal triangulation $\Delta$ on $\Sigma$. Recall that $V$ is the set of vertices in $\Delta$ (hence equal to the set of marked points), $V_{\circ}$ is the set of interior punctures, and $E$ is the set of edges in $\Delta$. We fix a total term order on $E$. Then via the generalized edge coordinate we introduced in Section \ref{ssec:coordinates}, we obtain the total order on the set $\RMC$.

The goal of this section is to prove that any central element $z \in Z(\cS_{q}^{\MRY}(\Sigma))$ can be given by the central elements in Section \ref{sec:central}. We first show that we may replace an arbitrary $z$ by a simpler element.

\begin{lemma}\label{lem:gradedpiecesarecentral}
Let $R = \bigoplus_{\bn \in \ZZ^{k}}R_{\bn}$ be a $\ZZ^{k}$-graded algebra and $z = \bigoplus_{\bn}z_{\bn} \in R$ with $z_{\bn} \in R_{\bn}$. Then $z \in Z(R)$ if and only if $z_{\bn} \in Z(R)$ for all $\bn \in \ZZ^{k}$.
\end{lemma}

\begin{proof}
Suppose $z_{\bn} \in Z(R)$ for all $\bn \in \ZZ^{k}$. Then for any $z' \in R$, $zz' = (\sum z_{\bn})z' = \sum z_{\bn}z' = \sum z'z_{\bn} = z'(\sum z_{\bn}) = z'z$ and $z \in Z(R)$.

Conversely, assume $z \in Z(R)$. Now take a homogeneous $z' \in R_{\bm}$. Then $0 = zz' - z'z = (\sum z_{\bn})z' - z'(\sum z_{\bn}) = \sum (z_{\bn}z' - z'z_{\bn})$. Since all $z_{\bn}z' - z'z_{\bn}$ are in the different graded parts, $z_{\bn}z' - z'z_{\bn} = 0$. Thus, $z_{\bn}$ is commutative with $z'$ for all $\bn$. Any element $z' \in R$ is a finite sum of homogeneous elements. Thus, we can conclude $z_{\bn} \in Z(R)$ for all $\bn$. 
\end{proof}

Note that $\cS_{q}^{\MRY}(\Sigma)$ is a $\ZZ^{V_{\circ}}$-graded algebra. For a central element $z \in Z(\cS_{q}^{\MRY}(\Sigma))$, it is sufficient to show that each graded piece can be generated by the central elements in Section \ref{sec:central}. So from now on, we will assume that $z$ is homogeneous with respect to the $\ZZ^{V_{\circ}}$-grading.

The next lemma tells us that we may assume that there is no contribution of vertex classes on $z$.

\begin{lemma}\label{lem:novertices}
Let $z \in Z(\cS_{q}^{\MRY}(\Sigma))$ be a homogeneous element with respect to the degree $\deg_{V}$. Then after multiplying by a Laurent monomial $m(v)$ in vertex classes, $z$ is a $\CC$-linear combination of diagrams in $\RMC$. In particular, $\deg_{V} z = \sum_{v \in W}\be_{v}$ for some $W \subset V_{\circ}$. 
\end{lemma}

\begin{proof}
 Let $\deg_{V}(z) = \sum a_{v}\be_{v}$. Then $z$ can be written uniquely as a linear combination $\sum c_{i}m_{i}(v)\alpha_{i}$, where $c_{i} \in \CC$, $\alpha_{i} \in \RMC$, and $m_{i}(v)$ is a Laurent monomial with respect to the vertex classes. Since each $\alpha_{i}$ is in $\RMC$, for each interior puncture $v$, there must be at most one arc component ending at $v$ and the degree is one or zero at $v$. So $\deg_{V}(\alpha_{i}) \in \{0, 1\}^{V_{\circ}}$. On the other hand, $\deg_{V}(m_{i}(v)) \in 2\ZZ^{V_{\circ}}$. Therefore, all $\alpha_{i}$ must have the same degree. This implies that $m_{i}(v)$ are all identical. Hence by multiplying $m_{i}(v)^{-1}$, we may assume that $z = \sum c_{i}\alpha_{i}$ and $\deg_{V} z = \sum_{v \in W}\be_{v}$ for some $W \subset V_{\circ}$. 
\end{proof}

\begin{lemma}\label{lem:GECisinjective}
Let $f : \RMC \to \frac{1}{2}\ZZ^{E}$ be the generalized edge coordinate function defined in Section \ref{ssec:coordinates}. If $\alpha, \alpha' \in \RMC$ and $f(\alpha) = f(\alpha')$, then $\alpha = \alpha'$. In other words, the edge coordinate defines the curve $\alpha \in \RMC$ uniquely.
\end{lemma}

\begin{proof}
Since the height information does not affect to show the claim, we ignore the information in the proof.

Let $\deg_{V}(\alpha) = \sum_{v \in W}\be_{v}$ for some $W \subset V_{\circ}$. The leading term $\lt(\alpha^{2}\prod_{v \in W}v)$ is obtained by the following recipe. For each component not meeting an interior puncture, square the component. For an arc component $\beta$ connecting two interior punctures, replace it by a loop $p_{\beta}$ surrounding $\beta$. For an arc component $\gamma$ connecting an interior puncture $v$ and a boundary marked point $w$, replace it by a boundary loop $r_{\gamma}$ of a once-punctured monogon surrounding $v$. By reversing this procedure, we have that the original curve $\alpha$ can be uniquely recovered from $\lt_{f}(\alpha^{2}\prod_{v \in W}v)$. From $f(\alpha) = f(\alpha')$, we know $\lt_{f}(\alpha^{2}\prod_{v \in W}v) = \lt_{f}(\alpha'^{2}\prod_{v \in W'}v)$. Since the corner coordinates for ordinary reduced multicurves uniquely recover the curve, we have $\lt_{f}(\alpha^{2}\prod_{v \in W}v) = \lt_{f}(\alpha'^{2}\prod_{v \in W'}v)$, hence $\alpha = \alpha'$. 
\end{proof}

Thus, we may define the total order on $\RMC$ by using $f : \RMC \to \frac{1}{2}\ZZ^{E}$ and any lexicographical order on $\frac{1}{2}\ZZ^{E}$.

\begin{lemma}\label{lem:edgecoordinequality}
Let $\alpha_{1}, \alpha_{2}  \in \RMC$ such that $\deg_{V}(\alpha_{1}) = \deg_{V}(\alpha_{2}) = \prod_{v \in W}\be_{v}$ for some $W \subset V_{\circ}$. Then for any resolution $\beta$ of $\alpha_{1}\alpha_{2}\prod_{v \in W}v$ using the relations (A) and (D) in Definition \ref{def:MRY}, 
\[
	f(\beta) \le f(\alpha_{1}) + f(\alpha_{2}).
\]
\end{lemma}

\begin{proof}
First, we take a `partial resolution' $\beta'$ of $\alpha_{1}\alpha_{2}\prod_{v \in W}v$, by choosing a resolution for each vertex $v \in W$. We show that $f(\beta') \le f(\alpha_{1}) + f(\alpha_{2})$. Note that the formula of $f(\beta')$ is obtained by looking at the contribution of its change on the neighbor of $v \in W$ ($E(v)$ in Remark \ref{rmk:formula}). We understood all curves after applying the moving trick, so none of the curves meet the boundary marked point.

We divide it into several cases.

\textsf{Case 1.} Near $v \in W$, $\alpha_{1}$ and $\alpha_{2}$ are not edges in $\Delta$. See Figure \ref{fig:comparison}. $f(\alpha_{1}) + f(\alpha_{2})$ provides one for each edge in $E(v)$, while each resolution contributes either 0 or 1 for each edge.

\textsf{Case 2.} Near $v \in W$, $\alpha_{1}$ is not an edge in $\Delta$, but $\alpha_{2}$ is an edge $s$ in $\Delta$. See Figure \ref{fig:comparison2}. In particular, one may see that in $f(\alpha_{1}) + f(\alpha_{2})$, the contribution of $E(v)$ of the edge coordinate for $s$ is $\frac{1}{2}$. On the other hand, for some resolution, the calculated intersection number along $s$ is (at most) one, so it looks larger than that of $f(\alpha_{1})+f(\alpha_{2})$. However, note that in this case, the other end $v'$ of $\alpha_{2}$ is also a marked point in $V$. If it is an interior puncture $v'$, then $v' \in W$. Since $\deg_{V}(\alpha_{1}) = \deg_{V}(\alpha_{2})$, the other end of $\alpha_{1}$ is at $v'$. Then by combining the contributions from $E(v)$ and $E(v')$, we still can obtain the desired inequality. If $v'$ is a boundary marked point, around $v'$, $mov(\alpha_{2})$ passes through the interior of an edge. See Figure \ref{fig:comparison5}. Then the calculation of the edge coordinate is reduced to \textsf{Case 1}.

\textsf{Case 3.} Near $v \in W$, both $\alpha_{1}$ and $\alpha_{2}$ are two distinct edges $s_{1}$, $s_{2}$ in $\Delta$. We may argue as in the previous case. See Figure \ref{fig:comparison3}.

\textsf{Case 4.} Near $v \in W$, $\alpha_{1}$, $\alpha_{2}$ are the same edge $s \in \Delta$. In this case, there is only one resolution that makes a non-periperal and non-contractible loop, and the contribution is equal to that of $f(\alpha_{1}) + f(\alpha_{2})$. Consult Figure \ref{fig:comparison4}.

Thus, for any case, $f(\beta') \le f(\alpha_{1}) + f(\alpha_{2})$. For any full resolution $\beta$ of $\beta'$ obtained by taking an interior resolution, we know $f(\beta) \le f(\beta')$. Thus, we obtain the desired result. 
\end{proof}

\begin{lemma}\label{lem:leadingterm}
Let $z = \sum c_{i}\alpha_{i}\in \cS_{q}^{\MRY}(\Sigma)$ be a homogeneous element of $\deg_{V}(z) = \sum_{v \in W}\be_{v}$ for some $W \subset V_{\circ}$ with a term order $f(\alpha_{1}) > f(\alpha_{2}) > \cdots > f(\alpha_{k})$. Then $\lt_{f}(z^{2}\prod_{v \in W}v) = \lt_{f}(\alpha_{1}^{2}\prod_{v \in W}v)$. 
\end{lemma}

\begin{proof}
It is sufficient to show that if $f(\alpha_{1}) > f(\alpha_{2})$, $\lt_{f}(\alpha_{1}^{2}\prod_{v \in W}v) > \lt_{f}(\alpha_{1}\alpha_{2}\prod_{v \in W}v)$. By definition, $\lt_{f}(\alpha_{1}^{2}\prod_{v \in W}v) = 2f(\alpha_{1})$. On the other hand, Lemma \ref{lem:edgecoordinequality} implies that for any term $\beta$ in $\alpha_{1}\alpha_{2}\prod_{v \in W}v$, $f(\beta) \le f(\alpha_{1}) + f(\alpha_{2})$. Thus, $\lt_{f}(\alpha_{1}\alpha_{2}\prod_{v \in W}v) \le f(\alpha_{1}) + f(\alpha_{2})$. Therefore, 
\[
	f(\alpha_{1}\alpha_{2}\prod_{v \in W}v) \le f(\alpha_{1}) + f(\alpha_{2}) < 2f(\alpha_{1}) = f(\alpha_{1}^{2}\prod_{v \in W}v), 
\]
where the second inequality holds from the assumption. 
\end{proof}

\begin{figure}
\includegraphics[width=0.6\textwidth]{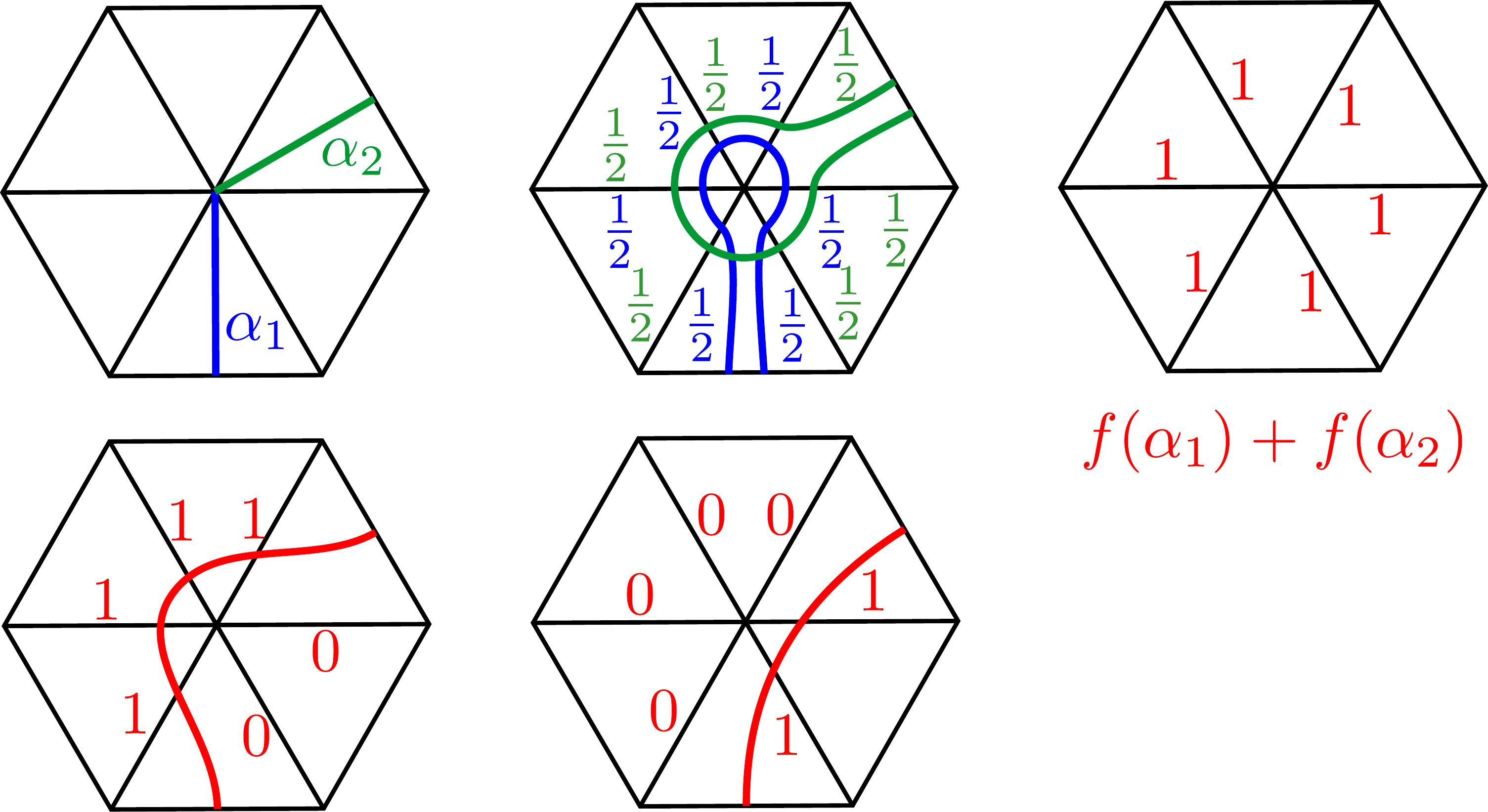}
\caption{Comparison of $f(\alpha_{1}) + f(\alpha_{2})$ and resolutions of $v\alpha_{1}\alpha_{2}$. The second row shows two resolutions of $v\alpha_{1}\alpha_{2}$ and their edge coordinates.}
\label{fig:comparison}
\end{figure}

\begin{figure}
\includegraphics[width=0.6\textwidth]{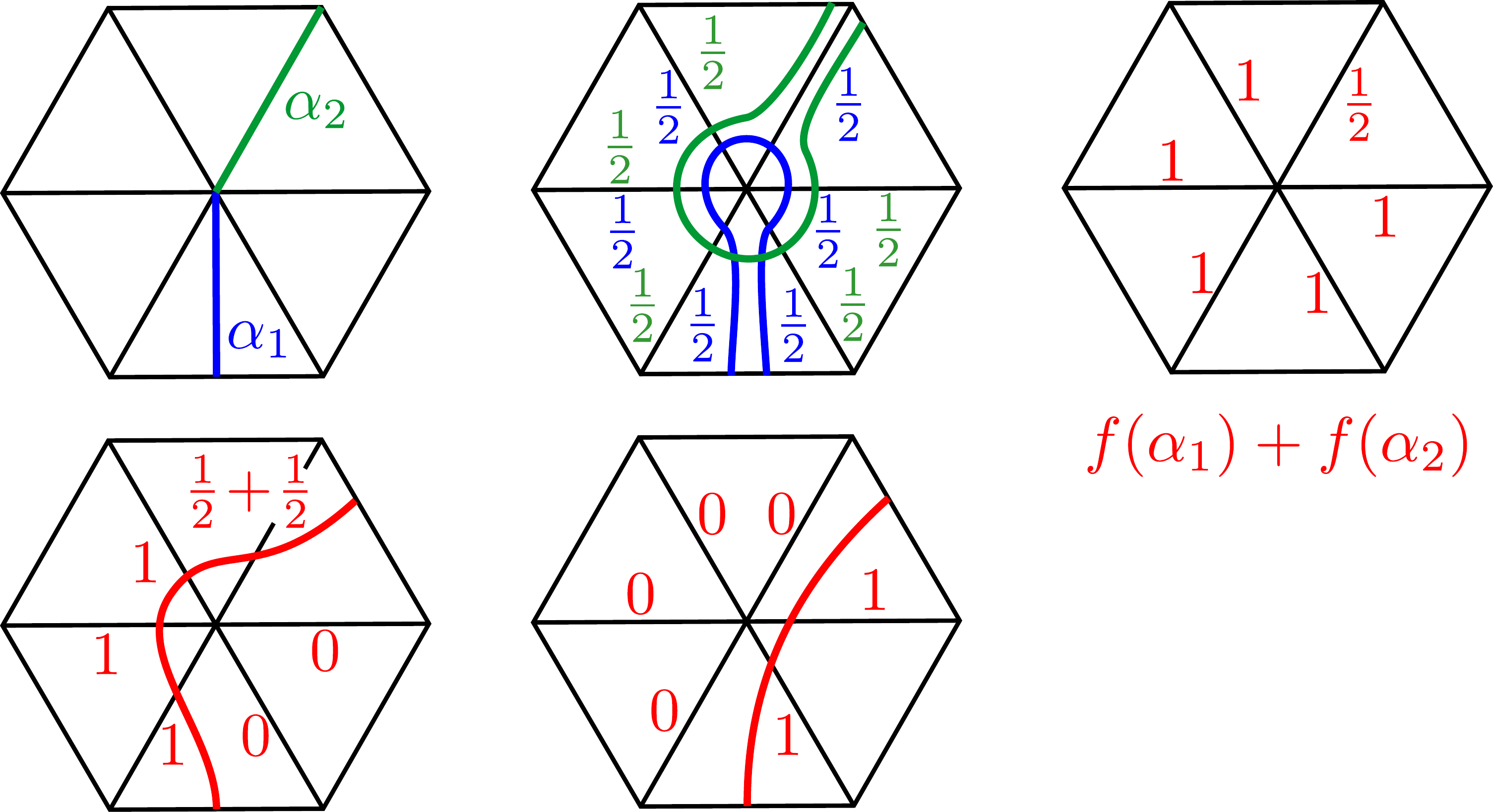}
\caption{Comparison of $f(\alpha_{1}) + f(\alpha_{2})$ and resolutions of $v\alpha_{1}\alpha_{2}$, in the case that $\alpha_{2}$ has an edge component. The second row shows two of resolutions of $v\alpha_{1}\alpha_{2}$ and their edge coordinates.}
\label{fig:comparison2}
\end{figure}

\begin{figure}
\includegraphics[width=0.6\textwidth]{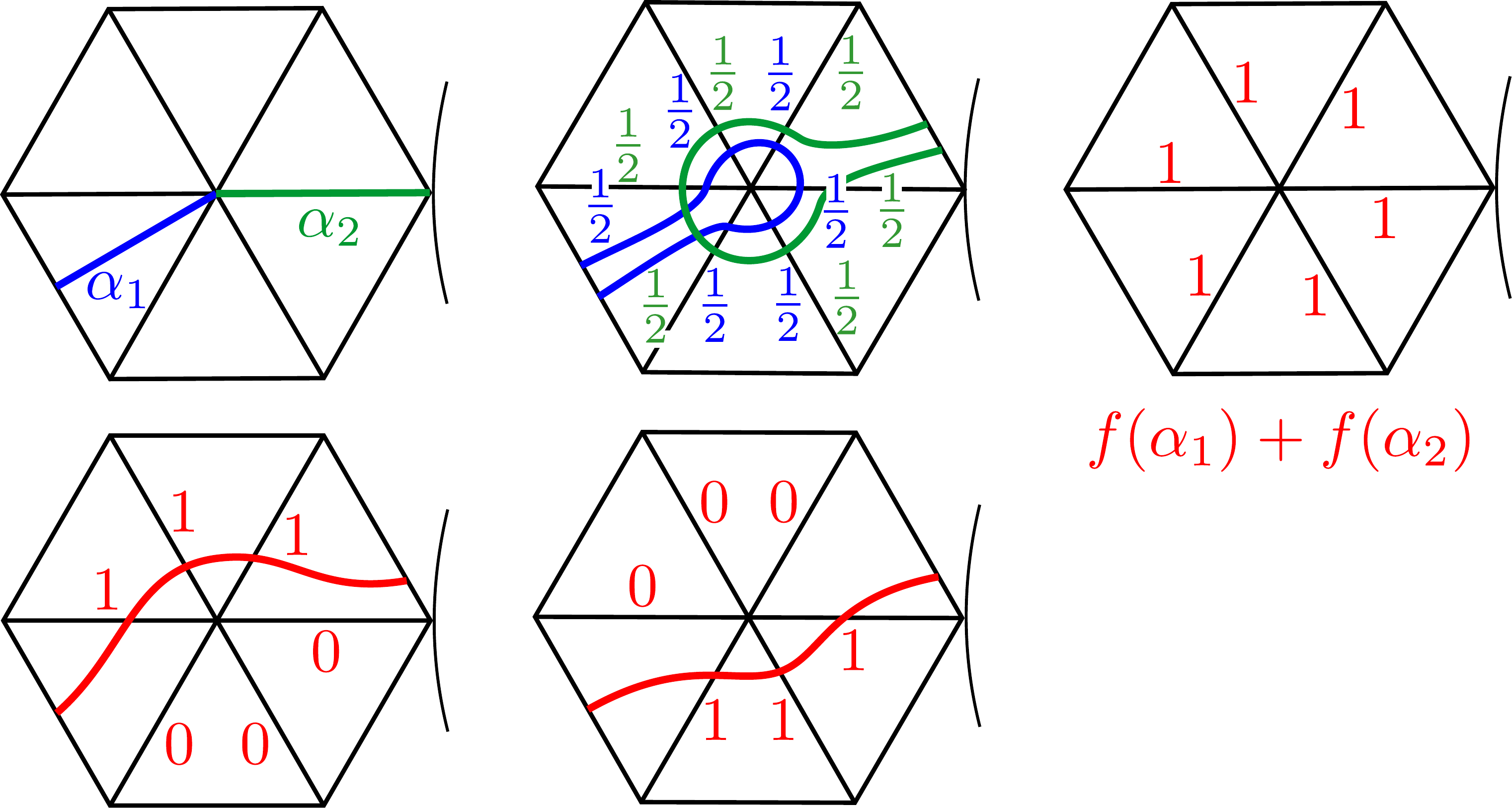}
\caption{Comparison of $f(\alpha_{1}) + f(\alpha_{2})$ and resolutions of $v\alpha_{1}\alpha_{2}$, in the case that one end of $\alpha_{2}$ is a boundary marked point.}
\label{fig:comparison5}
\end{figure}

\begin{figure}
\includegraphics[width=0.6\textwidth]{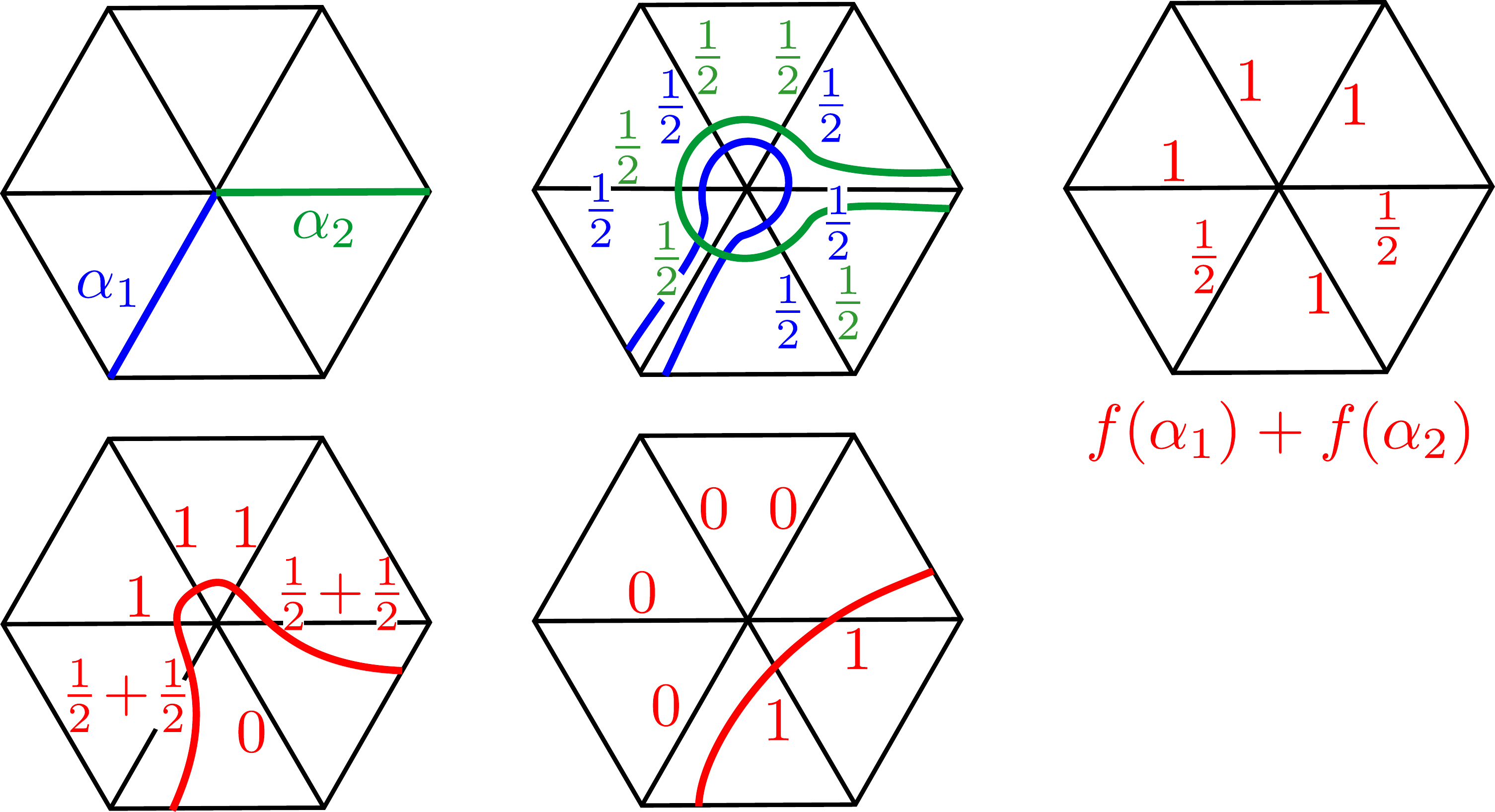}
\caption{Comparison of $f(\alpha_{1}) + f(\alpha_{2})$ and resolutions of $v\alpha_{1}\alpha_{2}$, in the case that both $\alpha_{1}$ and $\alpha_{2}$ have distinct edge components. The second row shows two of resolutions of $v\alpha_{1}\alpha_{2}$ and their edge coordinates.}
\label{fig:comparison3}
\end{figure}

\begin{figure}
\includegraphics[width=0.8\textwidth]{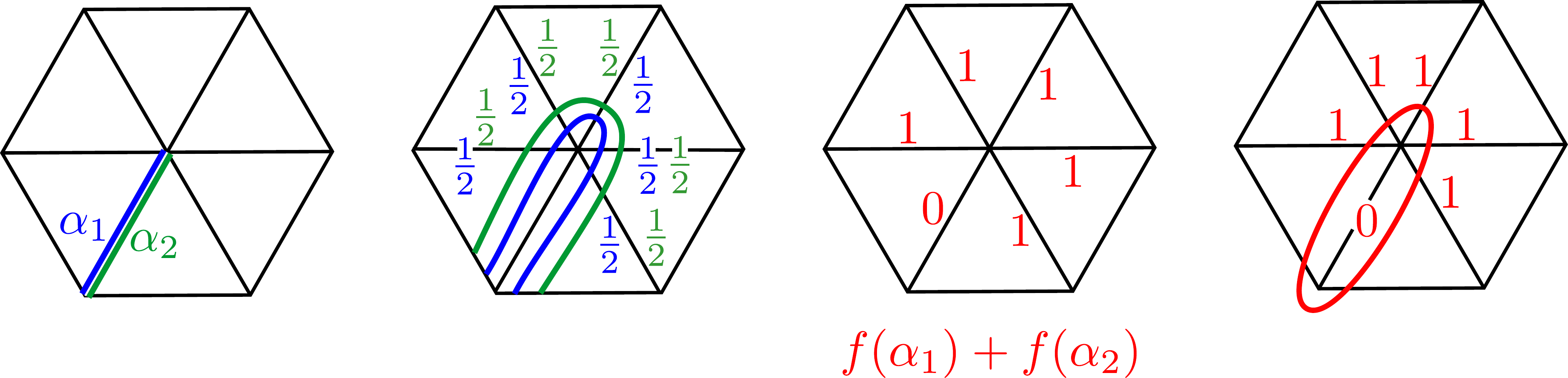}
\caption{Comparison of $f(\alpha_{1}) + f(\alpha_{2})$ and resolutions of $v\alpha_{1}\alpha_{2}$, in the case that both $\alpha_{1}$ and $\alpha_{2}$ have the same edge component. The second row shows the nonconstant resolution of $v\alpha_{1}\alpha_{2}$.}
\label{fig:comparison4}
\end{figure}

Our proof of Theorem \ref{thm:mainthm} relies on the computation of $Z(\overline{\cS}_{q}^{\rL+}(\Sigma))$ by Korinman in \cite{Kor21} that we restate here.

\begin{theorem}[\protect{\cite[Theorem 6.3]{Kor21}}]\label{thm:centerL+}
Let $q \in \CC^{*}$ be a primitive $n$-th root of unity for $n$ odd. The center of $\overline{\cS}_{q}^{\rL+}(\Sigma)$ is generated by the image of Chebyshev--Frobenius homomorphism $\Phi$, peripheral loops, the product of boundary classes $\beta_{D}$ for each boundary component $D \subset \partial \Sigma$ (Lemma \ref{lem:boudnaryarcs}), and $\beta_{D}^{-1}$.
\end{theorem}

We lift the description of the center to that of $\cS_{q}^{\rM+}(\Sigma)$.

\begin{lemma}\label{lem:centerbeforelocalization}
Let $R$ be a $\CC$-algebra which is a domain, and $S$ be a multiplicative subset of $R$. Let $S^{-1}R$ be the left Ore localization of $R$ by $S$. Then $Z(R) = Z(S^{-1}R) \cap R$. 
\end{lemma}

\begin{proof}
Suppose that $a \in Z(R)$. by the inclusion $R \hookrightarrow S^{-1}R$, we may identify $a$ with $1^{-1}a$. Take $s^{-1}r \in S^{-1}R$. Note that from $a \in Z(R)$, $as = sa$, so $as^{-1} = s^{-1}a$. By the definition of the multiplication in the non-commutative localization, 
\[
	(1^{-1}a)(s^{-1}r) = 1^{-1}s^{-1}ar = (s1)^{-1}ar = s^{-1}ra = (1s)^{-1}ra = s^{-1}1^{-1}ra = (s^{-1}r)(1^{-1}a).
\]
Therefore, $a  = 1^{-1}a \in Z(S^{-1}R)$. Hence $Z(R) \subset Z(S^{-1}R) \cap R$.

Conversely, suppose that $1^{-1}a \in Z(S^{-1}R) \cap R$. By the assumption, for any $1^{-1}b \in R$, $(1^{-1}a)(1^{-1}b) = (1^{-1}b)(1^{-1}a)$. But from $(1^{-1}a)(1^{-1}b) = 1^{-1}ab$ and $R \hookrightarrow S^{-1}R$, we obtain $ab = ba$ for all $b \in R$. Therefore, $a \in Z(R)$. 
\end{proof}

\begin{proposition}\label{prop:centerM+}
Let $q \in \CC^{*}$ be a primitive $n$-th root of unity for $n$ odd. The center of $\cS_{q}^{\rM+}(\Sigma)$ is generated by the image of the Chebyshev--Frobenius homomorphism $\Phi$, peripheral loops, and the product of boundary classes $\beta_{D}$ for each boundary component $D \subset \partial \Sigma$ (Lemma \ref{lem:boudnaryarcs}). 
\end{proposition}

\begin{proof}
We have $\overline{\cS}_{q}^{\rL+}(\Sigma) \cong \cS_{q}^{\rM+}(\Sigma)[\partial^{-1}]$ \cite{LY22}. Under this identification, the generating sets are the image of $\Phi : \cS_{1}^{\rM+}(\Sigma) \to \cS_{q}^{\rM+}(\Sigma)$, peripheral loops, and $\beta_{D}$, $\beta_{D}^{-1}$, and $\beta_{i}^{-n}$ for each boundary arc $\beta_{i}$. The image of $\Phi$, peripheral loops, $\beta_{D}$ are elements in $\cS_{q}^{\rM+}(\Sigma)$. Lemma \ref{lem:centerbeforelocalization} and the fact that $\cS_{q}^{\rM+}(\Sigma)$ is a domain, a corollary of \cite[Theorem 3]{LY22}, imply that $Z(\cS_{q}^{\rM+}(\Sigma))$ is generated by the described elements. 
\end{proof}

We are ready to prove Theorem \ref{thm:mainthm}.

\begin{proof}[Proof of Theorem \ref{thm:mainthm}]
Take a central element $z\in Z(\cS_{q}^{\MRY}(\Sigma))$. Applying Lemma \ref{lem:gradedpiecesarecentral}, we may assume that $z$ is homogeneous with respect to $\deg_{V}$. By multiplying some product of vertices, we may assume that $z = \sum_{i=1}^k c_i z_i$ with $c_i\in \CC$ and $z_i\in \RMC$. We may further assume that $f(z_1)>f(z_2)>\dots > f(z_k)$. Let $\deg_{V}(z) = \prod_{W \subset V_{\circ}}\be_{v}$. Consider $z^2\prod_{v \in W} v$. Then, $\lt_{f} (z^2\prod v)=\lt_{f}(z_1^2\prod v)$.

We may write $z_1 \in \RMC$ uniquely as
\begin{equation}\label{eqn:alpha}
	z_1 = \prod \alpha_{i}^{m_{i}} \prod p_{\beta_{j}}^{m_{j}}\beta_{j} \prod r_{\gamma_{k}}^{m_{k}}\gamma_{k} \prod r_{\delta_{\ell}}^{m_{\ell}} \prod {\eta_{s}}^{m_{s}}
\end{equation}
up to a power of $q$, where:
\begin{enumerate}
\item $\alpha_i$ is a loop class; 
\item $\beta_{j}$ is an arc class connecting two distinct interior punctures $v_{j}$ and $w_{j}$;
\item $p_{\beta_{j}}$ is a loop surrounding $\beta_{j}$;
\item $\gamma_{k}$, $\delta_{\ell}$ are arc classes connecting one interior puncture $x_{k}$ and one boundary marked point;
\item $r_{\gamma_{k}}$, $r_{\delta_{\ell}}$ are arc classes bounding a once-punctured monogon surrounding $\gamma_{k}$ and $\delta_{\ell}$ respectively, and; 
\item $\eta_{s}$ is an arc connecting two boundary marked points, 
\end{enumerate}
and none of the curves are isotopic to each other. All curves in the above description are disjoint and non-isotopic except at some of the boundary marked points. See Figure \ref{pic;ex_curves} for an example.

\begin{figure}[ht]\centering\includegraphics[width=170pt]{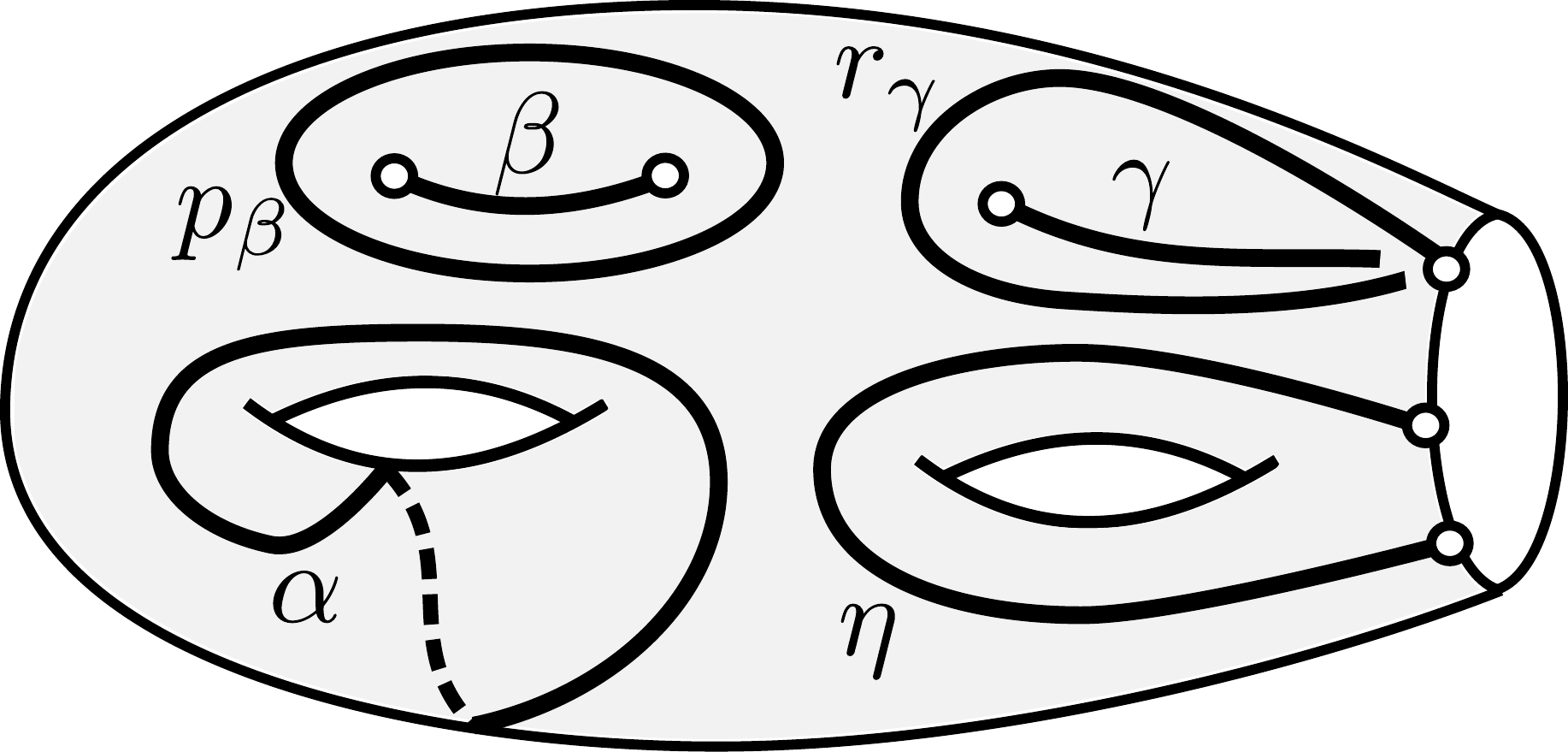}\caption{A presentation of the leading term in \eqref{eqn:alpha}}\label{pic;ex_curves}\end{figure}

In Lemma \ref{lem:leadingterm}, we showed that $\lt_{f}( z^{2}\prod_{v \in W}v) = \lt_{f}(z_1^{2}\prod_{v \in W}v )$. By a direct computation, we have 
\[
	z_1^{2}\prod_{v \in W}v  = \prod \alpha_{i}^{2m_{i}} \prod p_{\beta_{j}}^{2m_{j}}(p_{\beta_{j}}+2) \prod r_{\gamma_{k}}^{2m_{k}+1} \prod r_{\delta_{\ell}}^{2m_{\ell}} \prod {\eta_{s}}^{2m_{s}}.
\]
Thus, 
\begin{equation}\label{eqn:ltalpha}
	\lt_{f}( z_1^{2}\prod_{v \in W}v) = \prod \alpha_{i}^{2m_{i}} \prod p_{\beta_{j}}^{2m_{j}+1} \prod r_{\gamma_{k}}^{2m_{k}+1} \prod r_{\delta_{\ell}}^{2m_{\ell}} \prod {\eta_{s}}^{2m_{s}}
\end{equation}
and this is the leading term of $z^{2}\prod_{v \in W}v$.

On each boundary component $D$ of $\partial \Sigma$, we may find a maximum $x_{D} \in \ZZ_{\ge 0}$ such that $z_{1}$ is a multiple of $\beta_{D}^{x_{D}}$. Set $y_{D}$ as the unique integer such that $y_{D} \equiv x_{D} \mbox{ mod } n$ and $0 \le y_{D} < n$.

Since $z^{2}\prod_{v \in W}v \in Z(\cS_{q}^{\MRY}(\Sigma))$ and $\deg_\circ (z^{2}\prod_{v \in W}v )=\mathbf{0}$, we may regard its image as an element in $\cS_{q}^{\rM}(\Sigma)$, which is central. We may also think it as a central element in $\cS_{q}^{\rM+}(\Sigma)$ without any peripheral components. And $\lt_{f}(z_1^{2}\prod_{v \in W}v)$ is the leading term of a central element $z^{2}\prod_{v \in W}v$. By the above argument, $z_1^{2}\prod_{v \in W}v$ is generated by the image of Chebyshev--Frobenius homomorphism $\Phi$ and boundary central elements, i.e. an element of the form $\beta_{D}$ for some component $D \subset \partial \Sigma$. Then
\[
	\prod \alpha_{i}^{2m_{i}} \prod r_{\delta_{\ell}}^{2m_{\ell}} \prod {\eta_{s}}^{2m_{s}}\prod \beta_{D}^{-2y_{D}} 
\]
 is generated by the image of Chebyshev--Frobenius homomorphism $\Phi$.

If $\alpha'$ is the leading term of an element in the image of $\Phi$, then all of its edge coordinates are multiples of $n$. Because $n$ is odd, it also implies that all of the corner coordinates are also multiples of $n$. As all corner coordinates coming from $\prod \alpha_{i}^{2m_{i}} \prod r_{\delta_{\ell}}^{2m_{\ell}} \prod {\eta_{s}}^{2m_{s}}\prod \beta_{D}^{-2y_{D}}$ are multiples of $n$ and multiples of two, we may take a reduced multicurve $\alpha'$ whose corner coordinates are that of $\prod \alpha_{i}^{2m_{i}} \prod r_{\delta_{\ell}}^{2m_{\ell}} \prod {\eta_{s}}^{2m_{s}}\prod \beta_{D}^{-2y_{D}}$ divided by $2n$. Moreover, the corner coordinates coming from $\prod p_{\beta_{j}}^{2m_{j}+1}\prod r_{\gamma_{k}}^{2m_{k}+1}$ are also multiples of $n$. So we obtain $n|2m_{j}+1$ and $n|2m_{k}+1$. Then $2m_{j}+1 = n(2t_{j}+1)$ for some $t_{j} \in \ZZ$, hence $m_{j}$ can be written as $\frac{n-1}{2}+t_{j}n$. So is $m_{k}$. Now 
\[
\begin{split}
	\prod p_{\beta_{j}}^{m_{j}}\beta_{j}\prod r_{\gamma_{k}}^{m_{k}}\gamma_{k}
	&= \prod p_{\beta_{j}}^{\frac{n-1}{2}+t_{j}n}\beta_{j} \prod r_{\gamma_{k}}^{\frac{n-1}{2}+t_{k}n}\gamma_{k}\\ 
	&= \prod p_{\beta_{j}}^{\frac{n-1}{2}}\beta_{j} \prod (p_{\beta_{j}}^{n})^{t_{j}}\prod r_{\gamma_{k}}^{\frac{n-1}{2}}\gamma_{k}\prod (r_{\gamma_{k}}^{n})^{t_{k}}.
\end{split}
\]
The latter is indeed, up to a constant multiple, the leading term of 
\begin{equation}\label{eqn:zprime}
	z'' := \prod \frac{1}{\sqrt{v_{j}}\sqrt{w_{j}}}T_{n}(\sqrt{v_{j}}\sqrt{w_{j}}\beta_{j})T_{n}(p_{\beta_{j}})^{t_{j}} \prod \gamma_{k}^{n}\prod (r_{\gamma_{k}}^{n})^{t_{k}}.
\end{equation}

Now set $z' := T_{n}(\alpha') z''\prod \beta_{D}^{y_{D}}$. Then $z'$ is a product of central elements listed in Theorem~\ref{thm:mainthm}. By taking $z - z'$, we have a central element with a smaller leading term. The induction procedure terminates in finite steps since the edge coordinates are in $\ZZ_{\ge 0}^k$, and we obtain the desired result.
\end{proof}

The same argument indeed describe the center of $Z(\cS_{q}^{\MRY}(\Sigma)[\partial^{-1}])$. For the future reference (mostly because of its relationship with cluster algebra \cite{KMW25+}), we leave the statement and the proof.

\begin{theorem}\label{thm:localized}
Let $q \in \CC^{*}$ be a primitive $n$-th root of unity for odd $n$. The center $Z(\cS_{q}^{\MRY}(\Sigma)[\partial^{-1}])$ is a $\CC[v_{i}^{\pm}]$-subalgebra generated by the elements in the statement of Theorem \ref{thm:mainthm} and 
\begin{enumerate}
\item For a boundary arc class $\beta$, $\beta^{-n}$;
\item For each component $D$ of $\partial \Sigma$, $\beta_{D}^{-1}$. 
\end{enumerate}
\end{theorem}

\begin{proof}
For a central element $z\in Z(\cS_{q}^{\MRY}(\Sigma)[\partial^{-1}])$, by multiplying sufficient boundary central elements, denoted by $\alpha$, we may assume $\alpha z\in Z(\cS_{q}^{\MRY}(\Sigma))$. From the above argument, $z$ is generated by the central elements listed in Theorem~\ref{thm:mainthm}. Since $\alpha^{-1}\in Z(\cS_{q}^{\MRY}(\Sigma)[\partial^{-1}])$, we conclude $z=\alpha^{-1}(\alpha z)\in Z(\cS_{q}^{\MRY}(\Sigma)[\partial^{-1}])$ and $z$ is generated by central elements listed in Theorem~\ref{thm:mainthm} and the inverses of boundary central elements. 
\end{proof}


\section{Almost Azumaya algebras}\label{sec:almostAzumaya}
In this section, we investigate the representation theory of $\cS_{q}^{\MRY}(\Sigma)$.

\subsection{(Almost) Azumaya algebra and its representation theory}\label{ssec:Azumaya}

Let us first  review basic knowledge on (almost) Azumaya algebras following \cite{MR01}. See also \cite{FKBL19}.
Recall that a $\CC$-algebra $A$ is \emph{Azumaya} if $A$ is finitely generated projective $Z(A)$-module and the natural morphism 
\[
	A \otimes_{Z(A)}A^{op} \to \mathrm{End}_{Z(A)}A
\]
of $A$-modules is an isomorphism. By Artin-Wedderburn theorem,  if $Z(A) = \CC$, then $A \cong \rM_{n}(\CC)$ for some $n$. If $Z(A)$ is a finitely generated $\CC$-algebra, then over the associated affine algebraic variety $\mathrm{MaxSpec} \;Z(A)$, $A$ can be understood as a continuous family of matrix algebras. Indeed, there is a well-known correspondence between $\mathrm{MaxSpec}\; Z(A)$ and the set of maximal ideals in $Z(A)$. If we take a maximal ideal $m \subset Z(A)$, $Z(A)/m \cong \CC$ and its fiber $A \otimes_{Z(A)}Z(A)/m$ is a central simple algebra over $Z(A)/m \cong \CC$. In particular, $A \otimes_{Z(A)}Z(A)/m$ is isomorphic to a matrix algebra $\rM_{n}(\CC)$ for some $n$. Since $A$ is projective, $A \otimes_{Z(A)}Z(A)/m$ is locally free, hence $n$ is independent from the choice of $m$. The number $n$ is called the \emph{PI degree}.

A $\CC$-algebra $A$ is \emph{almost Azumaya} if there is an element $c \in Z(A)\setminus\{0\}$ such that its localization $A_{c}$ is Azumaya. For an affine algebraic variety, a localization corresponds to taking a Zariski open subset. Thus, being almost Azumaya implies that there is a Zariski open subset $U \subset \mathrm{MaxSpec}\; Z(A)$ such that $A_{c}$ is a family of matrix algebras over $U$. Note that if $Z(A)$ is an integral domain, $U$ is an open dense subset of $\mathrm{MaxSpec}\; Z(A)$.

The representation theory of an Azumaya algebra is very nice. Let $A$ be an Azumaya algebra with a finitely generated center $Z(A)$. Suppose that $V$ is a finite dimensional irreducible representation of $A$. Then the center $Z(A)$ acts as a scalar multiple on $V$. If we denote $m \subset Z(A)$ by the set of elements acting as the zero map is indeed a maximal ideal. Hence $V$ induces a representation of $A \otimes_{Z(A)}Z(A)/m \cong \rM_{n}(\CC)$. Since the only irreducible representation of $\rM_{n}(\CC)$ is $\CC^{n}$, $\dim V = n$. In summary, any irreducible representation of $A$ is of dimension $n$, and there is a bijection between the set $\mathrm{Irrep}(A)$ of finite dimensional irreducible $A$-representations and $\mathrm{MaxSpec}\; Z(A)$.

Suppose that $A$ is an almost Azumaya with a finitely generated center $Z(A)$. We retain the same notation in the above discussion. Let $V$ be a finite dimensional irreducible representation of $A$. As before, its center $Z(A)$ acts as a scalar multiple and we may choose a maximal ideal $m \subset Z(A)$ acting as zero on $V$. If $m \in U \subset \mathrm{MaxSpec}\; Z(A)$, then the localizing element $c \notin m$ and acts as a nonzero scalar. Thus, $V$ has a $A_{c}$-representation structure. Since $A_{c}$ is Azumaya, we know $\dim V = n$ and there is only one such a representation.

We summarize the above discussion.

\begin{proposition}
Let $A$ be an almost Azumaya algebra and $Z(A)$ is a finitely generated integral domain. Let $c \in Z(A)$ such that $A_{c}$ is Azumaya over $Z(A)_{c}$. Let $U = \mathrm{MaxSpec}\; Z(A)_{c} \subset \mathrm{MaxSpec}\; Z(A)$. Then the following holds. 
\begin{enumerate}
\item There is an injective map $\chi : U \to \mathrm{Irrep}(A)$. 
\item Any irreducible representation $V \in \chi(U)$ has the same dimension, that is the \emph{PI degree} of $A_{c}$.
\end{enumerate}
\end{proposition}

\begin{remark}
The map $\chi : U \to \mathrm{Irrep}(A)$ can be obtained as the following. For any maximal ideal $m \in U$, $A/mA \cong A_{c}/mA_{c} \cong \rM_{n}(\CC)$. Thus, up to isomorphism, there is a unique $n$-dimensional $A/mA$-representation $V$. The $A$-representation structure is induced by the epimorphism $A \to A/mA \cong \mathrm{End}_{\CC}V$. 
\end{remark}

\subsection{$\cS_{q}^{\MRY}(\Sigma)$ is almost Azumaya}

We now show that $\cS_{q}^{\MRY}(\Sigma)$ is almost Azumaya. In this section, we assume that our surface $\Sigma$ admits an ideal triangulation, hence it has at least one puncture or marked point, and $\chi(\Sigma) < 0$. Proposition \ref{prop:almostAzumaya} shows the statement for surfaces with nonempty boundaries (hence there is at least one marked point). 
For a closed surface without punctures, $\cS_q^\MRY(\Sigma)$ is the same as the ordinary skein algebra $\cS_q(\Sigma)$, and the result was shown in \cite{FKBL19}.  In the case of a marked surface without punctures, it was proven by \cite{Kor21} for the Muller skein algebra. 
Therefore, the only remaining cases are for surfaces without boundary but with interior punctures, hence $\cS_q^{\MRY}(\Sigma) = \cS_q^{\mathrm{RY}}(\Sigma)$. Remark \ref{rmk:approachviaorderlyfinitegeneration} suggests an alternative possible approach covering all cases including no boundary.

The following characterization of almost Azumaya algebras is particularly useful:
\begin{theorem}[\protect{\cite[III.1.7]{BG02}, \cite[Theorem 2.7]{FKBL19}}]\label{thm:almostAzumaya}
Let $A$ be a $\CC$-algebra. Suppose $A$ is 
\begin{enumerate}
    \item finitely generated as a $\CC$-algebra; 
    \item prime; 
    \item finitely generated as a $Z(A)$-module. 
\end{enumerate}
Then, $A$ is almost Azumaya.
\end{theorem}

For $\cS_{q}^{\MRY}(\Sigma)$ or $\cS_{q}^{\mathrm{RY}}(\Sigma)$, Item (1) is known in \cite{BKWP16, BKL24, KMW25+}. For $\Sigma = \Sigma_{0,n}$, a genus 0 surface with $n$ punctures but without any boundary, an explicit presentation is also known \cite{ACDHM21}. Item (2) is obtained in \cite{BKL24} as a consequence of an embedding into a quantum torus. By checking Item (3), we show that $\cS_{q}^{\MRY}(\Sigma)$ is almost Azumaya (Proposition \ref{prop:almostAzumaya}). 
The argument here is a minor variation of \cite[Section 6.2]{BKL24}, modifying the setting with vertex classes.

Let $P = [P_{ij}]$ be an $n \times n$ integral skew-symmetric matrix. We may define the the quantum torus $\TT(P)$ as a non-commutative $\CC$-algebra 
\begin{equation}\label{eqn:quantumtorus}
	\CC\langle x_{1}^{\pm}, x_{2}^{\pm}, \cdots, x_{n}^{\pm}\rangle /
	(x_{i}x_{j} = q^{P_{ij}}x_{j}x_{i}).
\end{equation}
We also obtain a positive subalgebra $\TT^{+}(P) \subset \TT(P)$ generated by $x_{1}, x_{2}, \cdots, x_{n}$.

We use the following result.

\begin{lemma}[\protect{\cite[Lemma 8.5]{KW24}}]\label{lem:almostAzumaya}
Let $P = [P_{ij}]$ be a skew-symmetric integral matrix. Let $A$ be a finitely generated $\CC$-algebra such that $\TT^{+}(P) \subset A \subset \TT(P)$. Then $A$ is almost Azumaya and the PI degree of $A$ is equal to that of $\TT(P)$. 
\end{lemma}

\begin{proposition}\label{prop:almostAzumaya}
Let $q$ be a root of unity. 
Let $\Sigma$ be a triangulable marked surface with at least one boundary marked point. Then:
\begin{enumerate}
\item There is a quantum torus $\TT(P)$ such that 
\begin{equation}\label{eqn:MRYisbetweenQTs}
	\TT^{+}(P) \subset \cS_{q}^{\MRY}(\Sigma) \subset \TT(P), 
\end{equation}
where $\TT^{+}(P)$ is the positive part of $\TT(P)$. 
\item The skein algebra $\cS_{q}^{\MRY}(\Sigma)$ is almost Azumaya and its PI-degree is equal to that of $\mathbb{T}(P)$. 
\end{enumerate}
\end{proposition}

\begin{proof}
The proof  is a minor variation of \cite[Section 6.2]{BKL24}. For the reader's convenience, here we highlight an overall strategy and where we need to make a change.

For $\cS_{q}^{\MRY}(\Sigma)$, we construct a quantum torus $\TT(P)$ such that the inclusions \eqref{eqn:MRYisbetweenQTs} hold. Lemma \ref{lem:almostAzumaya} guarantees being almost Azumaya.

If $\Sigma$ has no interior punctures, then $\cS_{q}^{\MRY}(\Sigma)$ is the Muller's skein algebra $\cS_{q}^{\rM}(\Sigma)$. In this case, Muller constructed the quantum torus $\TT(P)$ in \cite[Theorem 6.14]{Mul16}.

Next, we consider the case when $\Sigma$ has at least one interior puncture. The desired inclusions in \eqref{eqn:MRYisbetweenQTs} were shown in \cite[Equation (59)]{BKL24} without vertex classes. However, the proof almost works for our setting including vertex classes. To apply the argument of \cite[Section 6.2]{BKL24}, we modify the definition of the skew-symmetric matrix $P$ to obtain a quantum torus.

First, we take a specific triangulation $\lambda$ of $\Sigma$ as follows.  Fix a boundary marked point $p$.   For every interior puncture,  take an embedded path connecting $p$. We may take the paths in such a way that they do not intersect each other except at $p$. For each path, there is a unique ideal arc  with both endpoints at $p$  which encircles the path.  Cutting along $\Sigma$ along one of these ideal arcs removes  a once-punctured monogon. We can choose the ideal arcs so that they do not intersect each other except at $p$.   Then we extend this collection of arcs from $p$ to interior punctures and ideal arcs surrounding them to an ideal triangulation $\lambda$ of $\Sigma$.

Let $\bar{\lambda}$ denote the subset of $\lambda$ obtained by removing all the ideal arcs. 
Let $\{v_{i}\}$ be the set of interior punctures. We take $I := \bar{\lambda} \sqcup \{v_{i}\}$ as the set of generating variables of the quantum torus we will construct. Then, define an skew-symmetric bilinear form $P_{I}$ on $I$ by 
\[
	P_I(e,e'):=\#\begin{array}{c}\includegraphics[scale=0.25]{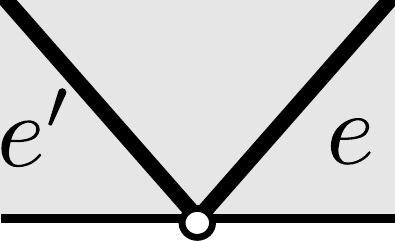}\end{array}
-\#\begin{array}{c}\includegraphics[scale=0.25]{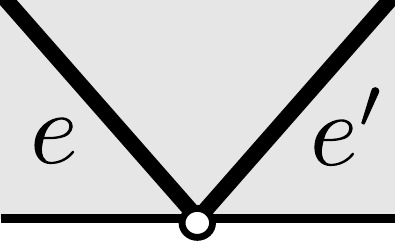}\end{array},\quad P_I(v_i, x)=0\text{ for any $x\in I$}, 
\]
where there are possibly other arcs between the arcs shown in the picture. 
Regarding $P_I$ as a matrix, we may construct a quantum torus $\TT(P_{I})$ associated to $P_{I}$.

Using the same argument as the proof of Theorem 6.5 in \cite{BKL24}, we obtain 
\[
	\TT^{+}(P_{I}) \subset \cS_{q}^{\MRY}(\Sigma) \subset \TT(P_{I}).\qedhere
\]
\end{proof}

\begin{remark}\label{rmk:approachviaorderlyfinitegeneration}

Thang L\^e kindly suggested the following alternative approach to show $\cS_q^\MRY(\Sigma)$ is almost Azumaya. As explained after Theorem \ref{thm:almostAzumaya}, it is sufficient to prove that $\cS_q^\MRY(\Sigma)$ is a finitely generated $Z(\cS_q^\MRY(\Sigma))$-module.

For a domain $R$, an $R$-algebra $A$ is \emph{orderly finitely generated} if there exists elements $a_1,\dots ,a_k\in A$ such that the set of all ordered monomials $\{ a_1^{n_1}\cdots  a_k^{n_k} \}_{n_i \in \ZZ_{\ge 0}}$ spans $A$ over $R$. In addition, if there is a polynomial $P_i$ for each $a_i$ such that $P_i(a_i)$ is central in $A$, then  one can decrease the degree on $a_i$ until it is equal to or less than the degree of $P_i$ and $A$ is finite dimensional as a $Z(A)$-module, as it is generated by $\{ a_1^{n_1}\cdots  a_k^{n_k} \}_{0 \le n_i < \deg P_i}$. Therefore, we obtain the finiteness  over $Z(A)$ and hence $A$ is almost Azumaya.

In \cite[Theorem 2 (b)]{BKL24}, it was proven that $\cS_{q}^{\rm RY}(\Sigma)$ is orderly finitely generated. One can also deduce it for  $\cS_{q}^{\rm MRY}(\Sigma)$ as follows. The LRY skein algebra is orderly finitely generated 
 \cite[Theorem 5 (c)]{BKL24}.  So its quotient, the reduced LRY skein algebra, is also orderly finitely generated.   
From \cite[Theorem 5.2]{LY22} or \cite{KMW25+}, the reduced LRY skein algebra is isomorphic to the 'boundary-localized' MRY skein algebra. The difference between $\cS_{q}^{\rm MRY}(\Sigma)$  and its boundary-localized version is invertible boundary edges, which are commutative with other elements up to a power of $q$.  Thus $\cS_{q}^{\rm MRY}(\Sigma)$ is also orderly finitely generated.

To apply the above strategy to show almost Azumaya in the case of skein algebras, we need that the generators are simple curves.  Suppose for the moment that we did have such a set, say $\{\alpha_1, \alpha_2, \cdots, \alpha_k\}$.   Then the threading operation is the same as multiplication, $\alpha_i^P = P(\alpha_i)$. And, as explained in Section \ref{sec:central}, we can use $T_n(x)$ or $x^n$ for $P$ to obtain central elements.  Then the argument above can be used to show that $\cS_q^{\mathrm{RY}}(\Sigma)$ is finite over $Z(\cS_q^{\mathrm{RY}}(\Sigma))$. 
At this point, we do not know an explicit orderly finitely generating set consisting of simple curves of $\cS_{q}^{\rm RY}(\Sigma)$. 
One might find such a generating set using admissibility arguments as in \cite{AF17} or the filtration in \cite{BKL24}. 
The same applies to $\cS_{q}^{\rm MRY}(\Sigma)$, but still our strategy of the proof of Proposition~\ref{prop:almostAzumaya} has an advantage -- it tells us what the PI-degree is. 
\end{remark}

Observe that the Artin-Tate lemma \cite[Lemma 13.9.10]{MR01} implies that under the assumptions of Theorem \ref{thm:almostAzumaya}, $Z(A)$ is a finitely generated $\CC$-algebra, thus $\mathrm{MaxSpec}\;Z(A)$ is an affine algebraic variety. Applying the discussion in Section \ref{ssec:Azumaya} to the skein algebra, it follows that there is a Zariski open subset $U$ of the variety $\mathrm{MaxSpec}\; Z(\cS_{q}^{\MRY}(\Sigma))$ where every point of $U$ corresponds to a unique finite-dimensional irreducible representation of $\cS_{q}^{\MRY}(\Sigma)$.  However, questions remain regarding the classification of such representations.

\begin{question} 
\begin{enumerate}
\item Describe the Azumaya locus $U \subset \mathrm{MaxSpec}\; Z(\cS_{q}^{\MRY}(\Sigma))$ explicitly. 
\item For the Azumaya points  $m \in U$, construct a finite dimensional irreducible representation $V$ of $\cS_{q}^{\MRY}(\Sigma)$ geometrically.
\item For the non-Azumaya points $m \notin U$, classify the corresponding finite dimensional irreducible representations of $\cS_{q}^{\MRY}(\Sigma)$. 
\end{enumerate}
\end{question}


\end{document}